\documentclass[11pt]{article}
\usepackage{eurosym}
\usepackage{amssymb}
\usepackage{amsfonts}
\usepackage{amsmath}
\usepackage{graphicx}
\usepackage{hyperref}
\usepackage{subcaption}
\usepackage{tikz}
\usepackage{float}
\usepackage{cite}
\usepackage{float}

\newtheorem{theorem}{Theorem}

\newtheorem{condition}[theorem]{Condition}

\newtheorem{lemma}[theorem]{Lemma}

\newtheorem{remark}[theorem]{Remark}

\newenvironment{proof}[1][Proof]{\textbf{#1.} }{\ \rule{0.5em}{0.5em}}
\begin{document}

\author{ Pelin G. Geredeli \thanks{%
email address: peling@iastate.edu.} \\
Department of Mathematics\\
Iowa State University, Ames-IA, USA}
\title{Bounded Semigroup Wellposedness for a Linearized Compressible Flow
Structure PDE Interaction with Material Derivative}
\maketitle

\begin{abstract}
We consider a compressible flow structure interaction (FSI) PDE system which
is linearized about some reference rest state. The deformable interface is
under the effect of an ambient field generated by the underlying and
unbounded material derivative term which further contributes to the
non-dissipativity of the FSI system, with respect to the standard energy
inner product. In this work we show that, on an appropriate subspace, only
one dimension less than the entire finite energy space, the FSI system is
wellposed, and is moreover associated with a continuous semigroup which is 
\emph{uniformly bounded} in time. Our approach involves establishing maximal
dissipativity with respect to a special inner product which is equivalent to
the standard inner product for the given finite energy space. Among other
technical features, the necesssary PDE estimates require the invocation of a
multiplier which is intrinsic to the given compressible FSI system.

\vskip.3cm \noindent \textbf{Key terms:} Flow-structure interaction,
compressible flows, wellposedness, uniformly bounded semigroup, material
derivative
\end{abstract}


\section{ Introduction}

Compressible flow phenomena arise in fluid mechanics, particularly in the
modeling of gas dynamics. The motion of such flows is typically described
via the Navier Stokes equations by way of providing qualitative information
on the three basic physical variables: the pressure of the fluid $p=p(x,t)$,
the mass density $\rho =\rho (x,t)$, the fluid velocity field $u=u(x,t)$.
Unlike the case of incompressible flows wherein density $\rho $ is a
constant, the pressure associated with compressible flow has a non-local
character and is an unknown function determined (implicitly) by the fluid
motion. Moreover, in compressible flow dynamics the density of the fluid is
considered to be an additional variable component, the resolution of which
represents substantial difficulties in the associated mathematical analysis.

In this work, we consider the linearization of a coupled
flow-structure-interaction (FSI) PDE system, with compressible fluid flow
PDE component. In the context of real world applications, this FSI finds its
key application in aeroelasticity: this PDE system involves the strong
coupling between a dynamically deforming structure (e.g. the wing) and the
air flow which streams past it. In short, this system describes the
interaction between plate and flow dynamics through a deformable interface.

The description of our FSI PDE model is given as follows: Let the flow
domain $\mathcal{O} \subset \mathbb{R}^{3}$ with boundary $\partial \mathcal{%
O}$. We assume that $\partial \mathcal{O}=\overline{S}\cup \overline{\Omega }
$, with $S\cap \Omega =\emptyset $, and with (structure) domain $\Omega
\subset \mathbb{R}^{3}$ being a \emph{flat} portion of $\partial \mathcal{O}$%
. In particular, $\partial \mathcal{O}$ has the following specific
configuration: 
\begin{equation}
\Omega \subset \left\{ x=(x_{1,}x_{2},0)\right\} \,\text{\ and \ surface }%
S\subset \left\{ x=(x_{1,}x_{2},x_{3}):x_{3}\leq 0\right\} \,.  \label{geo}
\end{equation}%

Let $\ \mathbf{n}(\mathbf{x})$ be the unit outward normal vector to $%
\partial \mathcal{O}$, and $\mathbf{n|}_{\Omega }=[0,0,1]$. Also, we denote
the unit outward normal vector to $\partial \Omega$ by $\mathbf{\nu}\mathbf{%
(x)}$. Additional geometric assumptions on $\mathcal{O}$ will be specified
later. Also, we assume that the pressure is a linear function of the
density; $p(x,t) = C\rho(x,t)$ as mostly done in the compressible fluid
literature and it is chosen as a primary variable to solve.

With respect to some equilibrium point of the form $\left\{ p_{\ast },%
\mathbf{U},\varrho _{\ast }\right\} $ where the pressure and density
components ${p_{\ast },\varrho _{\ast }}$ are assumed to be scalars, and the
arbitrary ambient field $\mathbf{U}:\mathcal{O}\rightarrow \mathbb{R}^{3}$ 
\begin{equation*}
\mathbf{U}%
(x_{1},x_{2},x_{3})=[U_{1}(x_{1},x_{2},x_{3}),U_{2}(x_{1},x_{2},x_{3}),U_{3}(x_{1},x_{2},x_{3})]
\end{equation*}%
is given, this linearization produces the following system of equations, in
solution variables $u(x_{1},x_{2},x_{3},t)$ (flow velocity), $%
p(x_{1},x_{2},x_{3},t)$ (pressure), $w_{1}(x_{1},x_{2},t)$ (elastic plate
displacement) and $w_{2}(x_{1},x_{2},t)$ (elastic plate velocity): 
\begin{align}
& \left\{ 
\begin{array}{l}
p_{t}+\mathbf{U}\cdot \nabla p+\text{div}~u\mathbf{+}\text{div}(\mathbf{U)}%
p=0~\text{ in }~\mathcal{O}\times (0,\infty ) \\ 
u_{t}+\mathbf{U}\cdot \nabla u-\text{div}\sigma (u)+\eta u+\nabla p=0~\text{
in }~\mathcal{O}\times (0,\infty ) \\ 
(\sigma (u)\mathbf{n}-p\mathbf{n})\cdot \boldsymbol{\tau }=0~\text{ on }%
~\partial \mathcal{O}\times (0,\infty ) \\ 
u\cdot \mathbf{n}=0~\text{ on }~S\times (0,\infty ) \\ 
u\cdot \mathbf{n}=w_{2}+\mathbf{U}\cdot \nabla w_{1}\text{ \ \ on }~\Omega
\times (0,\infty )\text{ }%
\end{array}%
\right.   \label{1} \\
& \left\{ 
\begin{array}{l}
w_{1_{t}}-w_{2}-\mathbf{U}\cdot \nabla w_{1}=0\text{ \ \ on }~\Omega \times
(0,\infty ) \\ 
w_{2_{t}}+\Delta ^{2}w_{1}+\left[ 2\nu \partial _{x_{3}}(u)_{3}+\lambda 
\text{div}(u)-p\right] _{\Omega }=0~\text{ on }~\Omega \times (0,\infty ) \\ 
w_{1}=\frac{\partial w_{1}}{\partial \nu }=0~\text{ on }~\partial \Omega
\times (0,\infty )%
\end{array}%
\right.   \label{2} \\
& 
\begin{array}{c}
\left[ p(0),u(0),w_{1}(0),w_{2}(0)\right] =\left[ p_{0},u_{0},w_{a},w_{b}%
\right] \in H_{N}^{\bot }.%
\end{array}
\label{3}
\end{align}%
where the space $H_{N}^{\bot }$ is defined in (\ref{null-ort}). The quantity 
$\eta >0$ represents a drag force of the domain on the viscous flow. In
addition, the quantity $\mathbf{\tau }$ in (\ref{1}) is in the space $%
TH^{1/2}(\partial \mathcal{O)}$ of tangential vector fields of Sobolev index
1/2; that is,%
\begin{equation}
\mathbf{\tau }\in TH^{1/2}(\partial \mathcal{O)=}\{\mathbf{v}\in \mathbf{H}^{%
\frac{1}{2}}(\partial \mathcal{O}):\mathbf{v}_{\partial \mathcal{O}}\cdot 
\mathbf{n}=0~\text{ on }~\partial \mathcal{O}\}.
\end{equation}%
(See e.g., p.846 of \cite{buffa2}.) In addition, we take ambient field $%
\mathbf{U}\in \mathbf{V}_{0}\cap W$ where 
\begin{equation}
\mathbf{V}_{0}=\{\mathbf{v}\in \mathbf{H}^{1}(\mathcal{O})~:~\left. \mathbf{v%
}\right\vert _{\partial \mathcal{O}}\cdot \mathbf{n}=0~\text{ on }~\partial 
\mathcal{O}\}  \label{V_0}
\end{equation}%
and 
\begin{equation}
W=\{v\in \mathbf{H}^{1}(\mathcal{O}):v\in L^{\infty }(\mathcal{O}),\text{ \
\ }div(v)\in L^{\infty }(\mathcal{O}),\text{ \ \ and \ \ }\mathbf{U}%
|_{\Omega }\in C^{2}(\overline{\Omega })\}  \label{W}
\end{equation}%
(This vanishing of the boundary for ambient fields is a standard assumption
in compressible flow literature; see \cite{dV},\cite{valli},\cite{decay},%
\cite{spectral}.) Moreover, the \textit{stress and strain tensors} in the
flow PDE component of (\ref{1})-(\ref{3}) are defined respectively as 
\begin{equation*}
\sigma (\mathbf{\mu })=2\nu \epsilon (\mathbf{\mu })+\lambda \lbrack
I_{3}\cdot \epsilon (\mathbf{\mu })]I_{3};\text{ \ }\epsilon _{ij}(\mathbf{%
\mu })=\dfrac{1}{2}\left( \frac{\partial \mathbf{\mu }_{j}}{\partial x_{i}}+%
\frac{\partial \mathbf{\mu }_{i}}{\partial x_{j}}\right) \text{, \ }1\leq
i,j\leq 3,
\end{equation*}%
where \textit{Lam\'{e} Coefficients }$\lambda \geq 0$ and $\nu >0$.

\begin{remark}
As will be seen below, the appearance of the term $-w_{2}-\mathbf{U}\cdot
\nabla w_{1}$, in the mechanical displacement equation (\ref{2}), will
induce an invariance with respect to the space $H_{N}^{\bot }$ defined in (%
\ref{null-ort}). We will ultimately establish that solutions of (\ref{1})-(%
\ref{3}), with initial data in $H_{N}^{\bot }$ , are associated with a
bounded semigroup, for $\mathbf{U}$ sufficiently small with respect to an appropriate measurement (see \ref{normU})). In addition, if we
set $w(t)=w_{1}(t)$, $w_{t}=w_{2}+\mathbf{U}\cdot \nabla w_{1}$, then we
have that $[p,u,w,w_{t}]$ solves 
\begin{align}
& \left\{ 
\begin{array}{l}
p_{t}+\mathbf{U}\cdot \nabla p+\text{div}~u\mathbf{+}\text{div}(\mathbf{U)}%
p=0~\text{ in }~\mathcal{O}\times (0,\infty ) \\ 
u_{t}+\mathbf{U}\cdot \nabla u-\text{div}\sigma (u)+\eta u+\nabla p=0~\text{
in }~\mathcal{O}\times (0,\infty ) \\ 
(\sigma (u)\mathbf{n}-p\mathbf{n})\cdot \boldsymbol{\tau }=0~\text{ on }%
~\partial \mathcal{O}\times (0,\infty ) \\ 
u\cdot \mathbf{n}=0~\text{ on }~S\times (0,\infty ) \\ 
u\cdot \mathbf{n}=w_{t}\text{ \ \ on }~\Omega \times (0,\infty )\text{ }%
\end{array}%
\right.  \notag \\
& \left\{ 
\begin{array}{l}
w_{tt}+\Delta ^{2}w-\mathbf{U}\cdot \nabla w_{t}+\left[ 2\nu \partial
_{x_{3}}(u)_{3}+\lambda \text{div}(u)-p\right] _{\Omega }=0~\text{ on }%
~\Omega \times (0,\infty ) \\ 
w=\frac{\partial w}{\partial \nu }=0~\text{ on }~\partial \Omega \times
(0,\infty )%
\end{array}%
\right.  \notag \\
& 
\begin{array}{c}
\left[ p(0),u(0),w(0),w_{t}(0)\right] =\left[ p_{0},u_{0},w_{a},w_{b}+%
\mathbf{U}\cdot \nabla w_{a}\right] \in H_{N}^{\bot }.%
\end{array}
\notag
\end{align}%
where $w(0)=w_{1}(0)=w_{a}$ and $w_{t}(0)=w_{2}(0)+\mathbf{U}\cdot \nabla
w_{1}(0)=w_{b}+\mathbf{U}\cdot \nabla w_{a}.$
\end{remark}

Here, as usually done for viscous fluids, we impose the so called \emph{%
impermeability condition} on $\Omega $; namely, we assume that no fluid
passes through the elastic portion of the boundary during deflection \cite%
{bolotin,dowell1}. At this point, we emphasize that the FSI problem under
consideration has present a \emph{material derivative} term on the deflected
interaction surface. This material derivative computes the time rate of
change of any quantity such as temperature or velocity (and hence also
acceleration) for a portion of a material in motion. Since our material is a
fluid, then the movement is simply the flow field and any particle of fluid
speeds up and down as it flows along the specified spatial domain. With
respect to the change of the speed of the said fluid, the material
derivative effectively gives a true rate of change of the velocity. Hence,
we describe the interface $\Omega $ in Lagrangian coordinates in $\mathbb{R}%
^{3}$ with $S(a_{1},a_{2},a_{3})=0$; also let $\mathbf{x}=\langle
x_{1},x_{2},x_{3}\rangle $ be the Eulerian position inside $\mathcal{O}$.
Then, letting $w(x_{1},x_{2},t)$ represent the transverse ($x_{3}$)
displacement of the plate on $\Omega $, we have that 
\begin{equation*}
S\big(x_{1},x_{2},x_{3}-w(x_{1},x_{2};t)\big)\equiv \mathcal{S}%
(x_{1},x_{2},x_{3};t)=0,
\end{equation*}%
describes the time-evolution of the boundary. The impermeability condition
requires that the material derivative ($\partial _{t}+\tilde{u}\cdot \nabla
_{\mathbf{x}}$) vanishes on the deflected surface \cite%
{bolotin,chorin-marsden,dowell1}: 
\begin{equation*}
\big(\partial _{t}\mathcal{+}\tilde{u}\cdot \nabla _{\mathbf{x}}\big)%
\mathcal{S}=0,~~~~~\tilde{u}=u+\mathbf{U}
\end{equation*}%
Applying the chain rule and rearranging, we obtain 
\begin{equation}
\nabla _{\mathbf{x}}S\cdot \langle 0,0,-w_{t}\rangle +\mathbf{U}\cdot
\lbrack \nabla _{\mathbf{x}}S+\langle
-S_{x_{3}}w_{x_{1}},-S_{x_{3}}w_{x_{2}},0\rangle ]=-u\cdot \lbrack \nabla _{%
\mathbf{x}}S+\langle -S_{x_{3}}w_{x_{1}},-S_{x_{3}}w_{x_{2}},0\rangle ].
\label{condish}
\end{equation}%
We identify $\nabla _{\mathbf{x}}S$ as the normal to the deflected surface; 
\emph{assuming small deflections} and restricting to $(x_{1},x_{2})\in
\Omega $, we can identify $\nabla _{\mathbf{x}}S\big|_{\Omega }$ with $%
\mathbf{n}\big|_{\Omega }=\langle 0,0,1\rangle $. Making use of %
\eqref{condish}, imposing that $\mathbf{U}\cdot \mathbf{n}=0$ on $\partial 
\mathcal{O}$ (see \eqref{V_0} and discussion), and discarding quadratic
terms, this relation allows us to write for $(x_{1},x_{2})\in \Omega $: 
\begin{equation*}
\mathbf{n}\cdot \langle 0,0,w_{t}\rangle +\mathbf{U}\cdot \langle
w_{x_{1}},w_{x_{2}},0\rangle =u\cdot \mathbf{n}.
\end{equation*}%
This yields the desired flow boundary condition 
\begin{equation}
u\cdot \mathbf{n}\big|_{\Omega }=w_{t}+\mathbf{U}\cdot \nabla w
\end{equation}%
in (\ref{1})$_{5}$ via the material derivative of the deflected elastic
interaction surface.

We note that the flow linearization is taken with respect to a general
inhomogeneous compressible Navier-Stokes system. However, unlike the papers 
\cite{agw, material} where some forcing and energy level terms in the
pressure and flow equations have been neglected, due to their relative
unimportance therein, in this present study, the particular energy level
term $\text{div}(\mathbf{U)}p$ in $(\ref{1})_{1}$ can not be neglected,
inasmuch as it plays a part in establishing that the associated FSI
semigroup is uniformly bounded (and invariant) with respect to the subspace $%
H_{N}^{\bot }$. Accordingly, the term $\text{div}(\mathbf{U)}p$ is one of
the ingredients in the \textquotedblleft feedback\textquotedblright\
operator $B$ defined in (\ref{feedbackB}).

In addition to the properties given for the fluid domain $\mathcal{O}$
before, we impose additional conditions which will be necessary for the
application of some elliptic regularity results for solutions of second
order boundary value problems on corner domains \cite{dauge, Dauge_3}:

\begin{condition}
\label{cond} Flow domain $\mathcal{O}$ should be curvilinear polyhedral
domain which satisfies the following condition:

\begin{itemize}
\item Each corner of the boundary $\partial \mathcal{O}$ -if any- is
diffeomorphic to a convex cone, \label{g1}

\item Each point on an edge of the boundary $\partial \mathcal{O}$ is
diffeomorphic to a wedge with opening $<\pi.$ \label{g2}
\end{itemize}
\end{condition}

\noindent Some examples of geometries can be seen in Figure 1.
\begin{figure}[]
\begin{subfigure}[H]{0.6\linewidth}
\centering
\begin{tikzpicture}[scale=0.9]
\draw[left color=black!10,right color=black!20,middle
color=black!50, ultra thick] (-2,0,0) to [out=0, in=180] (2,0,0)
to [out=270, in = 0] (0,-3,0) to [out=180, in =270] (-2,0,0);

\draw [fill=black!60, ultra thick] (-2,0,0) to [out=80,
in=205](-1.214,.607,0) to [out=25, in=180](0.2,.8,0) to [out=0,
in=155] (1.614,.507,0) to [out=335, in=100](2,0,0) to [out=270,
in=25] (1.214,-.507,0) to [out=205, in=0](-0.2,-.8,0) [out=180,
in=335] to (-1.614,-.607,0) to [out=155, in=260] (-2,0,0);

\draw [dashed, thin] (-1.7,-1.7,0) to [out=80, in=225](-.6,-1.3,0)
to [out=25, in=180](0.35,-1.1,0) to [out=0, in=155] (1.3,-1.4,0)
to [out=335, in=100](1.65,-1.7,0) to [out=270, in=25] (0.9,-2.0,0)
to [out=205, in=0](-0.2,-2.2,0) [out=180, in=335] to (-1.514,-2.0)
to [out=155, in=290] (-1.65,-1.7,0);

\node at (0.2,0.1,0) {{\LARGE$\Omega$}};

\node at (1.95,-1.5,0) {{\LARGE $S$}};

\node at (-0.3,-1.6,0) {{\LARGE $\mathcal{O}$}};
\end{tikzpicture}

 \end{subfigure}
\hfill 
\begin{subfigure}[H]{0.4\linewidth}
\centering
\includegraphics[width=\linewidth]{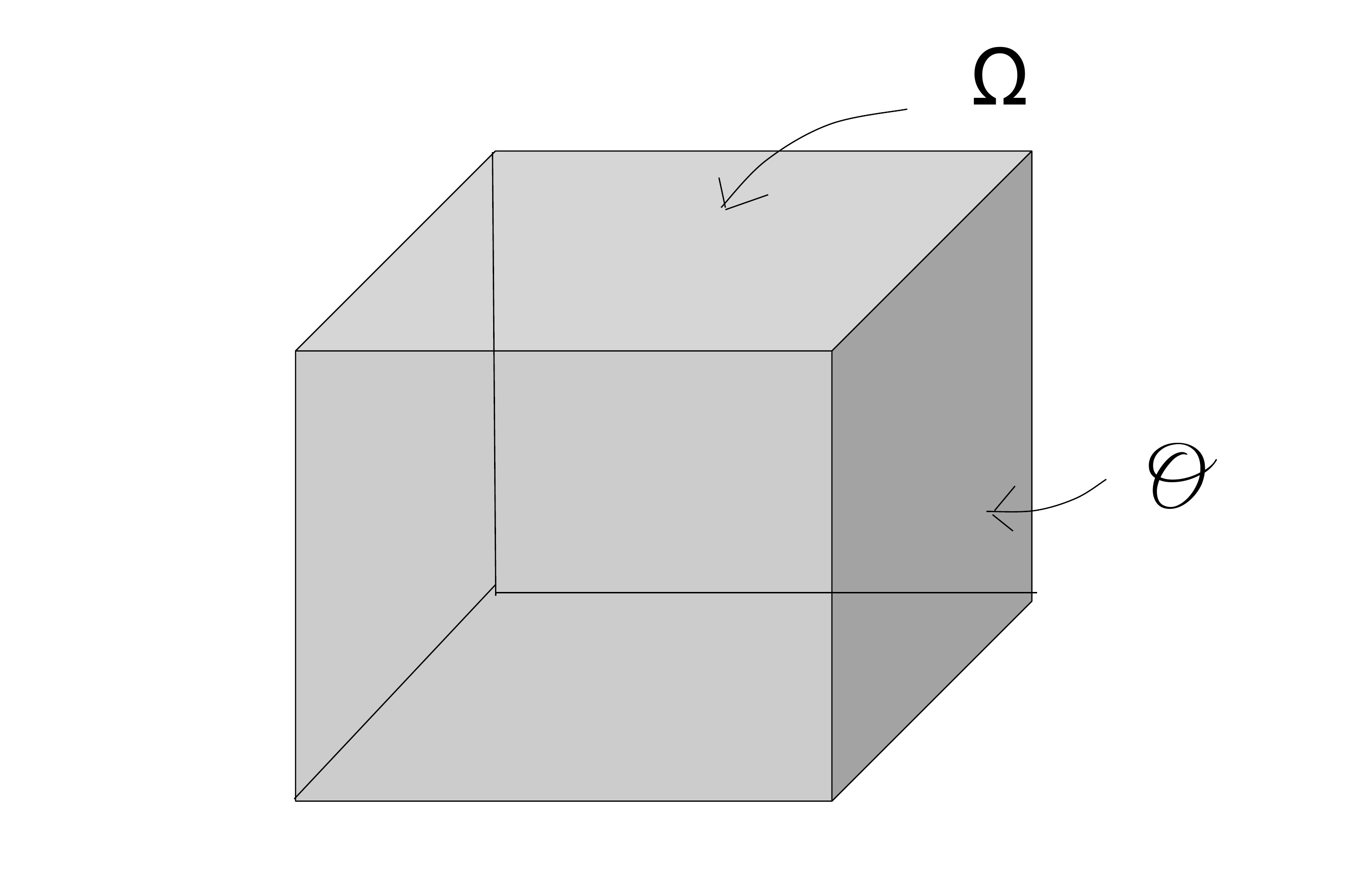}
\end{subfigure}
\caption{Polyhedral Flow-Structure Geometries }
\end{figure}
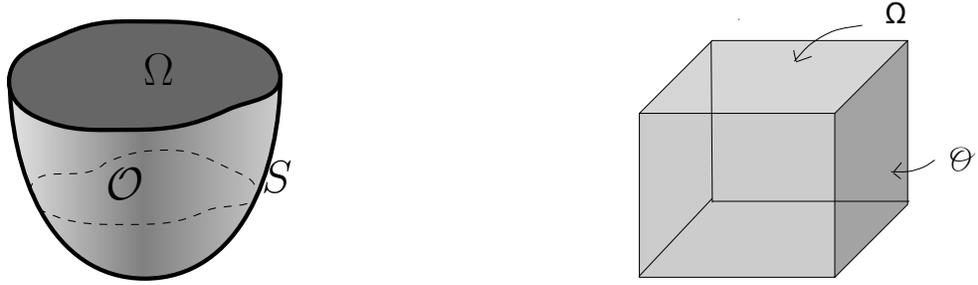
\noindent In reference to problem (\ref{1})-(\ref{3}), the associated finite
energy space will be 
\begin{equation}
\mathcal{H}\equiv L^{2}(\mathcal{O})\times \mathbf{L}^{2}(\mathcal{O})\times
H_{0}^{2}(\Omega )\times L^{2}(\Omega )  \label{stand}
\end{equation}%
which is a Hilbert space, topologized by the following standard inner
product: 
\begin{equation}
(\mathbf{y}_{1},\mathbf{y}_{2})_{\mathcal{H}}=(p_{1},p_{2})_{L^{2}(\mathcal{O%
})}+(u_{1},u_{2})_{\mathbf{L}^{2}(\mathcal{O})}+(\Delta w_{1},\Delta
w_{2})_{L^{2}(\Omega )}+(v_{1},v_{2})_{L^{2}(\Omega )}
\end{equation}%
for any $\mathbf{y}_{i}=(p_{i},u_{i},w_{i},v_{i})\in \mathcal{H},~i=1,2.$

\subsection{Literature}

The PDE's which describe fluid structure interactions have been considered
from a variety of viewpoints and with different objectives in mind; \cite%
{clark, dvorak, george1, george2, pelin-george, agw, agm, material, T1, T2,
ALT, lorena, igor, igor2, p1, canic}. Analysis of FSI generally constitutes
a broad area of research with applications in aeroelasticity, biomechanics,
biomedicine, etc. In particular, the study of wellposedness of various
linearized incompressible/compressible FSI models which manifest
parabolic-hyperbolic coupling has a large presence in the literature; see
e.g., \cite{clark, agw, agm, material, T1, T2, lorena, igor, canic} wherein
the Navier-Stokes equations are coupled with the wave/plate equation along a
fixed interface. The parabolic-hyperbolic nature of the system generally
results in major mathematical difficulties, principally because the coupling
mechanisms between the fluid and the solid PDE components inevitably
involves boundary terms which are strictly above the level of finite energy.
In the case of a \emph{compressible} flow component in the FSI system, the
analysis is further complicated: whereas for incompressible flows the
density of the fluid is assumed to be a constant and pressure an unknown
function determined by the fluid motion, for compressible flows the main
difficulty in the analysis of the density or pressure term, arises from the
fact that the density variable is no longer constant. Although in most of
the works in the literature, the motion of an isentropic compressible fluid
-- i.e., the density is a linear function of pressure -- is solely
considered, still, having to contend with this additional density (pressure)
variable presents a mathematical challenge, even at the level of
well-posedness.

In contrast to the growing literature on incompressible fluids the knowledge
about compressible fluids interacting with elastic solids is relatively
limited. In fact, the very first contribution to this problem is the
pioneering paper \cite{igor}, where both well-posedness and the existence of
global attractors were shown. In \cite{igor}, the author addresses the
simple case that the ambient vector field $\mathbf{U}=0$, i.e., i.e., the
linearization takes place about the trivial flow steady state. For this
canonical situation, he used Galerkin approximations to prove the
wellposedness result. However, the author duly noted that the case $\mathbf{U%
}\neq 0$ can not be handled in a similar fashion due to the existence of the
troublesome -- i.e., unbounded -- term $\mathbf{U\cdot }\nabla p$ in the
pressure equation $(\ref{1})_{1}.$

Subsequently, the linearized model in \cite{igor} with $\mathbf{U}\neq 0$
was considered in \cite{agw}. The linearization in \cite{agw}, about an
arbitrary non-zero state, gives rise to terms which induce a
non-dissipativity of the resulting FSI system. For this non-dissipative FSI
in \cite{agw}, a \emph{pure} velocity matching condition is imposed at the
interface (i.e., no material derivative is present in this boundary
condition). In contrast to the Galerkin approach applied in \cite{igor}, the
authors in \cite{agw} invoke a certain Lumer-Phillips methodology, with a
view of associating solutions of the fluid-structure dynamics with a
continuous semigroup which is not uniformly bounded. Subsequently, a more
convoluted FSI model was considered in \cite{material} where,
in addition to the aforesaid non-dissipative and unbounded terms brought
about by ambient field $\mathbf{U}\neq 0$, the associated flow-structure
interface is also under the effect of this ambient field $\mathbf{U}\neq 0$.
In particular, the flow and structure velocity matching boundary condition
also contains the \emph{material derivative} of the structure, which again
refers to the rate of change of the velocity on the deflected interaction
surface. In \cite{material} semigroup wellposedness is established by an
appropriate invocation of the Lumer-Phillips Theorem; this semigroup
generation is posed with respect to the \emph{entire} phase space $\mathcal{H%
}$, as defined in (\ref{stand}) above.

However, this wellposedness result in \cite{material} is not totally
satisfactory, from the standpoint of future studies into the long time
behavior of FSI solutions: while \cite{material} does provided existence and
uniqueness of solutions to the FSI system in the entire finite energy space $%
\mathcal{H}$, the resulting semigroup is \emph{not uniformly bounded}. In
particular, the semigroup estimate obtained in \cite{material} is $\mathcal{O%
}\left( e^{ {C(\mathbf{U})}t}\right) $, for $t>0$, where $C(\mathbf{U})=%
\frac{1}{2}\left\Vert \text{div}(\mathbf{U})\right\Vert _{\infty }+\epsilon $%
. This lack of FSI semigroup boundedness in \cite{material} will therefore
forestall any subsequent discussion of FSI stability. Accordingly, with a
mind toward future investigations of the asymptotic behavior of FSI
solutions, we are led to the following question: Is it possible to obtain a
semigroup wellposedness result, with the semigroup being bounded uniformly
in time, at least in some (inherently invariant) subspace of the finite
energy space?

\medskip

Motivated by this question, in the present work we consider the linearized
compressible flow-structure interaction model (\ref{1})-(\ref{3}), where $%
\mathbf{U}\neq 0$ and the material derivative term $\mathbf{U}\cdot \nabla
w_{1}$ is in place in the matching velocity boundary condition. Since our
main objective here is to obtain a \emph{uniformly bounded} semigroup, our
departure point is to find an appropriate subspace for the analysis. In
order to have semigroup generation on this sought-after subspace, the
prospective generator of the PDE system (\ref{1})-(\ref{3}) should be
invariant with respect to it. In this connection, it was shown in \cite{p1}
that if operator $\mathcal{A}_{0}:\mathcal{H}\rightarrow \mathcal{H}$ is the
FSI generator in \cite{agw}, which models the \textquotedblleft material
derivative\textquotedblright\ free FSI PDE interaction, then zero is an
eigenvalue of $\mathcal{A}_{0}$. (In particular, the action of $\mathcal{A}%
_{0}$ is given by $\mathcal{A}$ of (\ref{AAA}), with the appropriate domain
of definition [which includes the pure matching velocity boundary
condition]; see \cite{p1} and \cite{agw}). In fact, the null space of $%
\mathcal{A}_{0}$ is one dimensional, denoted here by $H_{N}$, and given
explicitly in (\ref{null}) below. The point of our mentioning $\mathcal{A}%
_{0}$ in the present problem is that, by way of obtaining a uniformly
bounded semigroup, we will take our candidate space of wellposedness to be
the orthogonal complement $H_{N}^{\bot }$, which is characterized by (\ref%
{null-ort}) below.

The necessity of finding an appropriate invariant subspace for uniformly
bounded FSI semigroup analysis motivates the presence of the additional (and
unbounded) term $w_{2}+\mathbf{U}\cdot \nabla w_{1}$ in (\ref{1})-(\ref{3}).
Let $\mathcal{A}_{1}:\mathcal{H}\rightarrow \mathcal{H}$ be the FSI
generator which gives rise to the wellposedness result in \cite{material};
the action of $\mathcal{A}_{1}$ is given by $\mathcal{A}$ of (\ref{AAA})
with the appropriate domain of definition, which includes the material
derivative term matching velocity boundary condition; see p. 342 of \cite%
{material}. As thus constituted,  $H_{N}^{\bot }$ is \emph{not} invariant
with respect to $\mathcal{A}_{1}$. However, if we define an operator $B$
which abstractly models the unbounded term $w_{2}+\mathbf{U}\cdot \nabla
w_{1}$ in (\ref{1})-(\ref{3}), as well as the energy level term $\text{div}(%
\mathbf{U})p\,$, then with the appropriate domain of definition, $%
H_{N}^{\bot }$ \emph{is }-invariant with respect to the modeling operator $(%
\mathcal{A}+B)$ of (\ref{1})-(\ref{3}. (This is Lemma 3 below).

Having established said invariance, we will subsequently proceed to show
that, with respect to a certain inner product which is equivalent to the
standard $\mathcal{H}$-inner product, $(\mathcal{A}+B)$ generates a
contraction semigroup on $H_{N}^{\bot }$, for ambient field $\mathbf{U}$
small enough in norm (and so the semigroup will be uniformly bounded with
respect to the standard $\mathcal{H}$-norm). In consequence, the PDE system (%
\ref{1})-(\ref{3}) is wellposed for initial data $[p_{0},u_{0},w_{a},w_{b}]$
taken from $H_{N}^{\bot }$.

\subsection{Challenges encountered and Novelty}

In the present work, we establish a result of semigroup wellposedness so as
to ascertain the existence and uniqueness of solutions to (\ref{1})-(\ref{3}%
), for Cauchy data in $H_{N}^{\bot }$. Moreover, we find this FSI semigroup
is uniformly bounded in time. This boundedness will have implications in our
future analysis of long time behavior of the solutions to the PDE system (%
\ref{1})-(\ref{3}). The main challenging points and improvements in our
treatment are as follows:\newline
\newline
\textbf{(a)}  \textit{Uniformly bounded semigroup in $H_{N}^{\bot }\subset 
\mathcal{H}:$} By way of fulfilling our objective of obtaining a uniformly
bounded semigroup, we adopt a Lumer-Phillips approach, in an appropriate
inner product. To wit, to establish dissipativity we topologize the $(%
\mathcal{A}+B)$-invariant space $H_{N}^{\bot }$ with an inner product which
is equivalent to the standard $\mathcal{H}$-inner product. In this
construction, we make use of a multiplier $\nabla \psi $ introduced in \cite%
{igor} (defined in (\ref{Igormap}) below) and previously used in \cite{p1};
the multiplier exploits the characterization of $H_{N}^{\bot }$ in (\ref%
{null-ort}) below. In addition, inasmuch as we are after a FSI solution
semigroup which is uniformly bounded in time, we give a proof for the
maximality (or the range condition) of the operator $(\mathcal{A}+B)$ which
is quite different than that in \cite{material}. Unlike \cite{material}
where the theory of linear perturbations is used so as to yield a semigroup
whose bound is of said exponential order, in the present we totally eschew
the Lax-Milgram approach of \cite{material} and instead invoke functional
analytical and PDE methods to show that $[\lambda I-(\mathcal{A}+B)]$ is
invertible for any $\lambda >0$. This entails to show that $[\lambda I-(%
\mathcal{A}+B)]$ is a closed linear operator that has a dense range in $%
H_{N}^{\bot }$ and enjoys the inverse estimate (\ref{36}) below. By these
means we establish that $(\mathcal{A}+B)$ is maximal dissipative with
respect to said appropriate inner product, and so then a uniformly bounded
semigroup on the standard $\mathcal{H}$-inner product. Our uniformly bounded
semigroup result is valid under the assumption that ambient vector field $%
\mathbf{U}$ is small enough with respect to an appropriate measurement; see (%
\ref{normU}) below. However, one should bear in mind that the present of $%
\mathbf{U}\neq 0$ gives rise to terms -- namely, $\mathbf{U}\cdot \nabla p$
and $\mathbf{U}\cdot \nabla w_{1}$ (as it appears twice) -- which are \emph{%
unbounded} with respect to the underlying finite energy of the FSI system.
Thus, our method of proof does not at all involve some bounded perturbation
result which exploits the smallness of $\mathbf{U}$. \newline
\newline
\textbf{(b)} $H_{N}^{\bot }$\textit{- invariant generator:} Subsequent to
our work \cite{material}, our original immediate objective was to analyze
the stability properties of the material derivative FSI system in \cite%
{material}. However, because of the presence of the zero eigenvalue, as
mentioned above, it is problematic to consider the strong or exponential
decay problem in the entire phase space $\mathcal{H}$. Accordingly, we are
led here to consider wellposedness (and future stability) analysis on $%
H_{N}^{\bot }$ as given in (\ref{null-ort}) below.(Since $H_{N}$ of (\ref%
{null}) is only one dimensional, --see \cite[Lemma 6]{p1}-- we would not
lose too much.) However, as we said above, $H_{N}^{\bot }$ is not invariant
with respect to the material derivative FSI generator $\mathcal{A}_{1}:%
\mathcal{H}\rightarrow \mathcal{H}$ in \cite{material}.  (The unbounded
material derivative term in particular contributes to the non-invariance.)
However, the presence of the terms $-w_{2}-\mathbf{U}\cdot \nabla w_{1}$ and 
$\text{div}(\mathbf{U})p$ in the respective structural displacement and
pressure equations in (\ref{1})-(\ref{3}) gives rise to an invariance on 
$H_{N}^{\bot }$. (Actually, the term $\text{div}(\mathbf{U})p$  was blithely
disgarded during the linearization process in \cite{material}, since it is a
benign energy level term.) Thus, these two terms are captured abstractly by
the \textquotedblleft feedback\textquotedblright\ operator $B$ in (\ref%
{feedbackB}) below. We say feedback, since $B$ is incorporated so as to
beneficently provide the pre-requisite that $H_{N}^{\bot }$ $\ $is $(%
\mathcal{A}+B)$-invariant$.$ We note that the presence of $B$ does \emph{not}
at all give rise to a fortuitous cancellation of terms so as to have
dissipativity with respect to the standard $\mathcal{H}$-inner product. The
operator $B$ allows only for said invariance property, so that our
wellposedness and uniform bounded semigroup problem can be considered on the
slightly smaller subspace $H_{N}^{\bot }$. As we said, our finding that the
FSI semigroup is uniformly bounded in time in $H_{N}^{\bot }$ will
constitute a departure point in our future work on stability properties of
the FSI PDE model.\newline
\newline
\textbf{(c) }\textit{Less regularity required on the ambient vector field $%
\mathbf{U}:$} The presence of the nontrivial ambient flow field $\mathbf{U}$
causes substantial difficulties in the wellposedness analysis. In this case $%
\mathbf{U}\neq 0$, the desired result for a FSI system -- with material
derivative present in the matching velocities BC -- on the entire phase
space $\mathcal{H}$ was obtained in the earlier work \cite{material} (with
recall, the semigroup estimate $\mathcal{O%
}\left( e^{ {C(\mathbf{U})}t}\right) $, for $t>0$, where $C(\mathbf{U})=\frac{1}{2}\left\Vert \text{div%
}(\mathbf{U})\right\Vert _{\infty }+\epsilon $). In the course of applying
the Lax-Milgram Theorem in \cite{material}, there is the need to deal with
the pressure PDE component of an associated static compressible FSI system.
In this regard, a methodology, based upon a treatment of (uncoupled)
transport equations in \cite{dV}, was applied to solve for the pressure and
fluid velocity components of said static FSI system. However this approach
compelled the authors in \cite{material} to impose that $\mathbf{U}\in 
\mathbf{H}^{3}(\mathcal{O})$. In the present work, we require that small
enough ambient field $\mathbf{U}\in \mathbf{H}^{1}(\mathcal{O})$ obey the
less stringent regularity assumptions in (\ref{W}).

\subsection{Notation}

\ Throughout, for a given domain $D$, the norm of corresponding space $%
L^{2}(D)$ will be denoted as $||\cdot ||_{D}$ (or simply $||\cdot ||$ when
the context is clear). Inner products in $L^{2}(\mathcal{O})$ or $\mathbf{L}%
^{2}(\mathcal{O})$ will be denoted by $(\cdot ,\cdot )_{\mathcal{O}}$,
whereas inner products $L^{2}(\partial \mathcal{O})$ will be written as $%
\langle \cdot ,\cdot \rangle _{\partial \mathcal{O}}$. We will also denote
pertinent duality pairings as $\left\langle \cdot ,\cdot \right\rangle
_{X\times X^{\prime }}$ for a given Hilbert space $X$. The space $H^{s}(D)$
will denote the Sobolev space of order $s$, defined on a domain $D$; $%
H_{0}^{s}(D)$ will denote the closure of $C_{0}^{\infty }(D)$ in the $%
H^{s}(D)$-norm $\Vert \cdot \Vert _{H^{s}(D)}$. We make use of the standard
notation for the boundary trace of functions defined on $\mathcal{O}$, which
are sufficently smooth: i.e., for a scalar function $\phi \in H^{s}(\mathcal{%
O})$, $\frac{1}{2}<s<\frac{3}{2}$, $\gamma (\phi )=\phi \big|_{\partial 
\mathcal{O}},$ which is a well-defined and surjective mapping on this range
of $s$, owing to the Sobolev Trace Theorem on Lipschitz domains (see e.g., 
\cite{necas}, or Theorem 3.38 of \cite{Mc}).

\subsection{Plan of the paper}

The paper is organized as follows: In Section 2, we first provide the
framework which will be required for our proof of semigroup wellposedness.
In particular, we carefully describe the FSI generator $(\mathcal{A}+B)$ and
its domain, as well as the equivalent inner product which will be used for
our proof of wellposedness on subspace $H_{N}^{\bot }$ of (\ref{null-ort})
below. Moreover, we show that $H_{N}^{\bot }$ is $(\mathcal{A}+B)$%
-invariant. In Section 3, we establish the maximal dissipativity of $(%
\mathcal{A}+B)$ with respect to said special inner product, thereby allowing
for an appeal to the Lumer-Phillips Theorem. In the course of our work, we
will have need of a classic lemma of functional analysis, as well as
the adjoint of $(\mathcal{A}+B)$.  These ingredients are given in Section 4,
the Appendix.

\section{Functional Setting and Preliminaries}

With respect to the above setting, the PDE system given in (\ref{1})-(\ref{3}%
) can be written as an ODE in Hilbert space $\mathcal{H}.$ That is, if $\Phi
(t)=\left[ p,u,w_1,w_{2}\right] \in C([0,T];\mathcal{H})$ solves the problem
(\ref{1})-(\ref{3}), then there is a modeling operator $\mathcal{A}+B:D(%
\mathcal{A}+B)\subset \mathcal{H}\rightarrow \mathcal{H}$ such that $\Phi
(\cdot )$ satisfies 
\begin{eqnarray}
\dfrac{d}{dt}\Phi (t) &=&(\mathcal{A}+B)\Phi (t);  \notag \\
\Phi (0) &=&\Phi _{0}  \label{ODE}
\end{eqnarray}%
Here the operators $\mathcal{A}$ and the feedback operator $B$ are defined
as follows:

\begin{equation}
\mathcal{A}=\left[ 
\begin{array}{cccc}
-\mathbf{U}\mathbb{\cdot }\nabla (\cdot ) & -\text{div}(\cdot ) & 0 & 0 \\ 
-\mathbb{\nabla (\cdot )} & \text{div}\sigma (\cdot )-\eta I-\mathbf{U}%
\mathbb{\cdot \nabla (\cdot )} & 0 & 0 \\ 
0 & 0 & 0 & I \\ 
\left. \left[ \cdot \right] \right\vert _{\Omega } & -\left[ 2\nu \partial
_{x_{3}}(\cdot )_{3}+\lambda \text{div}(\cdot )\right] _{\Omega } & -\Delta
^{2} & 0%
\end{array}%
\right] ;  \label{AAA}
\end{equation}%
and 
\begin{equation}
B=\left[ 
\begin{array}{cccc}
-\text{div}(\mathbf{U)(\cdot )} & 0 & 0 & 0 \\ 
0 & 0 & 0 & 0 \\ 
0 & 0 & \mathbf{U}\mathbb{\cdot }\nabla (\cdot ) & 0 \\ 
0 & 0 & 0 & 0%
\end{array}%
\right] .  \label{feedbackB}
\end{equation}%
\noindent Then, $D(\mathcal{A}+B)\subset \mathcal{H}$ is given by

\begin{equation*}
D(\mathcal{A}+B)=\{(p_{0},u_{0},w_{1},w_{2})\in L^{2}(\mathcal{O})\times 
\mathbf{H}^{1}(\mathcal{O})\times H_{0}^{2}(\Omega )\times L^{2}(\Omega )~:~%
\text{properties }(A.i)\text{--}(A.vi)~~\text{hold}\},
\end{equation*}%
where

\begin{enumerate}
\item[\textbf{(A.i)}] $\mathbf{U}\cdot \nabla p_{0}\in L^{2}(\mathcal{O})$

\item[\textbf{(A.ii)}] $\text{div}~\sigma (u_{0})-\nabla p_{0}\in \mathbf{L}%
^{2}(\mathcal{O})$ (So, $\left[ \sigma (u_{0})\mathbf{n}-p_{0}\mathbf{n}%
\right] _{\partial \mathcal{O}}\in \mathbf{H}^{-\frac{1}{2}}(\partial 
\mathcal{O})$)

\item[\textbf{(A.iii)}] $-\Delta ^{2}w_{1}-\left[ 2\nu \partial
_{x_{3}}(u_{0})_{3}+\lambda \text{div}(u_{0})\right] _{\Omega
}+p_{0}|_{\Omega }\in L^{2}(\Omega )$ (by elliptic regularity theory $w_1
\in H^{3}(\Omega))$

\item[\textbf{(A.iv)}] $\left( \sigma (u_{0})\mathbf{n}-p_{0}\mathbf{n}%
\right) \bot ~TH^{1/2}(\partial \mathcal{O})$. That is, 
\begin{equation*}
\left\langle \sigma (u_{0})\mathbf{n}-p_{0}\mathbf{n},\mathbf{\tau }%
\right\rangle _{\mathbf{H}^{-\frac{1}{2}}(\partial \mathcal{O})\times 
\mathbf{H}^{\frac{1}{2}}(\partial \mathcal{O})}=0\text{ \ in }\mathcal{D}%
^{\prime }(\mathcal{O})\text{\ for every }\mathbf{\tau }\in
TH^{1/2}(\partial \mathcal{O})
\end{equation*}

\item[\textbf{(A.v)}] $w_{2}+\mathbf{U}\cdot \nabla w_{1}\in
H_{0}^{2}(\Omega )$ (and so $w_{2}\in H_{0}^{1}(\Omega ))$

\item[\textbf{(A.vi)}] The flow velocity component $u_{0}=\mathbf{f}_{0}+%
\widetilde{\mathbf{f}}_{0}$, where $\mathbf{f}_{0}\in \mathbf{V}_{0}$ and $%
\widetilde{\mathbf{f}}_{0}\in \mathbf{H}^{1}(\mathcal{O})$ satisfies%
\footnote{%
The existence of an $\mathbf{H}^{1}(\mathcal{O})$-function $\widetilde{%
\mathbf{f}}_{0}$ with such a boundary trace on Lipschitz domain $\mathcal{O}$
is assured; see e.g., Theorem 3.33 of \cite{Mc}.}%
\begin{equation*}
\widetilde{\mathbf{f}}_{0}=%
\begin{cases}
0 & ~\text{ on }~S \\ 
(w_{2}+\mathbf{U}\cdot \nabla w_{1})\mathbf{n} & ~\text{ on}~\Omega%
\end{cases}%
\end{equation*}%
\noindent (and so $\mathbf{f}_{0}|_{\partial \mathcal{O}}\in
TH^{1/2}(\partial \mathcal{O})$).\newline
\end{enumerate}

Moreover, we denote 
\begin{equation}
H_{N}=Span\left\{ \left[ 
\begin{array}{c}
1 \\ 
0 \\ 
{{\mathring{A}}^{-1}(1)} \\ 
0%
\end{array}%
\right] \right\} ,  \label{null}
\end{equation}%
where $\mathring{A}:L^{2}(\Omega )\rightarrow L^{2}(\Omega )$ is the
elliptic operator 
\begin{equation*}
\mathring{A}\varpi =\Delta ^{2}\varpi \text{, with }D(\mathring{A})=\{w\in
H_{0}^{2}(\Omega ):\Delta ^{2}w\in L^{2}(\Omega )\},
\end{equation*}%
and%
\begin{equation}
H_{N}^{\bot }=\mathcal{\{}[p_{0},u_{0},w_{1},w_{2}]\in \mathcal{H}%
:\int\limits_{\mathcal{O}}p_{0}d\mathcal{O}+\int\limits_{\Omega
}w_{1}d\Omega =0\mathcal{\}}\text{\ \ }  \label{null-ort}
\end{equation}%
(see \cite[Lemma 6]{p1}).

As stated before, in order to be able to obtain a uniformly bounded
(contraction) semigroup, we analyze the wellposedness of problem (\ref{1})-(%
\ref{3}) in the reduced space $H_{N}^{\bot }$. This will require us to
re-topologize the phase space $\mathcal{H}$ with a new inner product to be
used in $H_{N}^{\bot }$ and equivalent to the natural inner product given in
(\ref{stand}). Now, with the above notation let us take $\varphi =\left[
p_{0},u_{0},w_{1},w_{2}\right] \in H_{N}^{\bot },$ $\widetilde{\varphi }=%
\left[ \widetilde{p}_{0},\widetilde{u}_{0},\widetilde{w}_{1},\widetilde{w}%
_{2}\right] \in H_{N}^{\bot }.$ Then the new inner product is given as 
\begin{equation*}
((\varphi ,\widetilde{\varphi }))_{H_{N}^{\bot }}=(p_{0},p_{0})_{\mathcal{O}%
}+(u_{0}-\alpha D(g\cdot \nabla w_{1})e_{3}+\xi \nabla \psi (p_{0},w_{1}),%
\widetilde{u}_{0}-\alpha D(g\cdot \nabla \widetilde{w}_{1})e_{3}+\xi \nabla
\psi (\widetilde{p}_{0},\widetilde{w}_{1}))_{\mathcal{O}}
\end{equation*}%
\begin{equation}
+(\Delta w_{1},\Delta \widetilde{w}_{1})_{\Omega }+(w_{2}+h_{\alpha }\cdot
\nabla w_{1}+\xi w_{1},\widetilde{w}_{2}+h_{\alpha }\cdot \nabla \widetilde{w%
}_{1}+\xi \widetilde{w}_{1})_{\Omega },  \label{innpro}
\end{equation}%
and in turn the norm 
\begin{equation*}
\left\Vert \left\vert \varphi \right\vert \right\Vert _{H_{N}^{\bot }}=\sqrt{%
\left( \left( \varphi ,\varphi \right) \right) _{H_{N}^{\bot }}}
\end{equation*}%
\begin{equation}
=\left\Vert p_{0}\right\Vert _{\mathcal{O}}^{2}+\left\Vert u_{0}-\alpha
D(g\cdot \nabla w_{1})e_{3}+\xi \nabla \psi (p_{0},w_{1})\right\Vert _{%
\mathcal{O}}^{2}+\left\Vert \Delta w_{1}\right\Vert _{\Omega
}^{2}+\left\Vert w_{2}+h_{\alpha }\cdot \nabla w_{1}+\xi w_{1}\right\Vert
_{\Omega }^{2}  \label{specnorm}
\end{equation}%
for every $\varphi =\left[ p_{0},u_{0},w_{1},w_{2}\right] \in H_{N}^{\bot }.$
Here, \newline
\newline
\textbf{(i)} the function $\psi =\psi (f,g)\in H^{1}(\mathcal{O)}$ is
considered to solve the following BVP for data $f\in L^{2}(\mathcal{O})$ and 
$g\in L^{2}(\Omega )$

\begin{equation}
\left\{ 
\begin{array}{c}
-\Delta \psi =f\text{ \ \ \ in \ }\mathcal{O} \\ 
\frac{\partial \psi }{\partial n}=0\text{ \ \ on \ }S \\ 
\frac{\partial \psi }{\partial n}=g\text{ \ \ on \ }\Omega 
\end{array}%
\right.   \label{Igormap}
\end{equation}%
with the compatibility condition%
\begin{equation}
\int\limits_{\mathcal{O}}fd\mathcal{O+}\int\limits_{\Omega }gd\Omega =0.
\label{Igormap1}
\end{equation}%
We should note that by known elliptic regularity results for the Neumann
problem on Lipschitz domains--see e.g; \cite{JK}-- we have%
\begin{equation}
\left\Vert \psi (f,g)\right\Vert _{H^{\frac{3}{2}}(\mathcal{O)}}\leq \left[
\left\Vert f\right\Vert _{\mathcal{O}}+\left\Vert g\right\Vert _{\partial 
\mathcal{O}}\right] .  \label{psireg}
\end{equation}%
\textbf{(ii)} the map $D(\cdot )$ is the Dirichlet map that extends boundary
data $\varphi $ defined on $\Omega $ to a harmonic function in $\mathcal{O}$
satisfying:%
\begin{equation*}
D\varphi =f\Leftrightarrow \left\{ 
\begin{array}{c}
\Delta f=0\text{ \ \ in \ }\mathcal{O} \\ 
f|_{\mathcal{\partial O}}=\varphi |_{ext}\text{ \ \ on \ \ }\mathcal{%
\partial O}%
\end{array}%
\right. 
\end{equation*}%
where 
\begin{equation*}
\varphi |_{ext}=%
\begin{cases}
0 & \text{ on }~S \\ 
\phi  & \text{ on }~\Omega 
\end{cases}%
\end{equation*}%
Then by, e.g., \cite[Theorem 3.3.8]{Mc}, and Lax-Milgram, we deduce that 
\begin{equation}
D\in \mathcal{L}\big(H_{0}^{1/2+\epsilon }(\Omega );H^{1}(\mathcal{O})\big).
\label{liftnorm}
\end{equation}%
\textbf{(iii)} the vector field $h_{\alpha }(\cdot )$ is defined as $%
h_{\alpha }(\cdot )=\mathbf{U}|_{\Omega }-\alpha g,$ where $g(\cdot )$ is a $%
C^{2}$ extension of the normal vector $\mathbf{n}(x)$ (with respect to $%
\Omega )$ and we specify the parameter $\alpha $ to be 
\begin{equation}
\alpha =2\left\Vert \mathbf{U}\right\Vert _{\ast },  \label{alpha}
\end{equation}%
where 
\begin{equation}
\left\Vert \mathbf{U}\right\Vert _{\ast }=\left\Vert \mathbf{U}\right\Vert _{L^{\infty }(%
\mathcal{O})}+\left\Vert \text{div}(\mathbf{U)}\right\Vert _{L^{\infty }(%
\mathcal{O})}+\left\Vert \mathbf{U}|_{\Omega }\right\Vert _{C^{2}(\overline{%
\Omega })}.  \label{normU}
\end{equation}%
Also, $\xi $ is eventually specified in (\ref{33}). Since the main goal of
this manuscript is to have the semigroup wellposedness in the subspace $%
H_{N}^{\bot }$, in what follows, for the sake of simplicity, we will use the
notation 
\begin{equation*}
(\mathcal{A}+B)|_{H_{N}^{\bot }}=(\mathcal{A}+B).
\end{equation*}%
Before beginning our wellposedness analysis, we firstly need to justify that
the semigroup generator is indeed $H_{N}^{\bot }-$ invariant. This is given
in the following lemma:

\begin{lemma}
\label{invariant} The operator $(\mathcal{A}+B)$ is $H_{N}^{\bot }-$
invariant; that is $(\mathcal{A}+B):D(\mathcal{A}+B)\cap H_{N}^{\bot
}\subset H_{N}^{\bot }\rightarrow H_{N}^{\bot }.$
\end{lemma}

\begin{proof}
Let $\varphi =\left[ p_{0},u_{0},w_{1},w_{2}\right] \in H_{N}^{\bot },$ $%
\widetilde{\varphi }=\left[ \widetilde{p}_{0},\widetilde{u}_{0},\widetilde{w}%
_{1},\widetilde{w}_{2}\right] \in H_{N}.$ Recalling the adjoint operator $%
\mathcal{A}^{\ast }$ in (\ref{adj-AplusB}) we have%
\begin{equation*}
(\mathcal{A}\varphi ,\widetilde{\varphi })_{\mathcal{H}}=(\varphi ,\mathcal{A%
}^{\ast }\widetilde{\varphi })_{\mathcal{H}}=(\varphi ,L_{1}\widetilde{%
\varphi })_{\mathcal{H}}+(\varphi ,L_{2}\widetilde{\varphi })_{\mathcal{H}%
}=0+(\varphi ,L_{2}\widetilde{\varphi })_{\mathcal{H}}
\end{equation*}

\begin{equation*}
=\int\limits_{\mathcal{O}}p_{0}\text{div}(\mathbf{U)}1d\mathcal{O+}%
\int\limits_{\Omega }\Delta w_{1}\Delta {{\mathring{A}}^{-1}}\left\{ \text{%
div}[U_{1},U_{2}]\right\} 1d\Omega
\end{equation*}%
\begin{equation*}
=\int\limits_{\mathcal{O}}p_{0}\text{div}(\mathbf{U)}1d\mathcal{O+}%
\int\limits_{\Omega}w_{1}\text{div}[U_{1},U_{2}]1d\Omega
\end{equation*}%
\begin{equation*}
=\int\limits_{\mathcal{O}}\text{div}(\mathbf{U)}p_{0}1d\mathcal{O-}%
\int\limits_{\Omega}(\nabla w_{1}\cdot \mathbf{U)}1d\Omega
\end{equation*}%
\begin{equation*}
=\int\limits_{\mathcal{O}}\text{div}(\mathbf{U)}p_{0}1d\mathcal{O-}%
\int\limits_{\Omega}\Delta (\nabla w_{1}\cdot \mathbf{U)}\Delta {{\mathring{A%
}}^{-1}(}1)d\Omega
\end{equation*}%
\begin{equation*}
=\left( \left[ 
\begin{array}{c}
\text{div}(\mathbf{U)}p_{0} \\ 
0 \\ 
-\nabla w_{1}\cdot \mathbf{U} \\ 
0%
\end{array}%
\right] ,\left[ 
\begin{array}{c}
1 \\ 
0 \\ 
{{\mathring{A}}^{-1}(}1) \\ 
0%
\end{array}%
\right] \right) _{\mathcal{H}}
\end{equation*}%
\begin{equation*}
=-\left( B\varphi ,\widetilde{\varphi }\right) _{\mathcal{H}}
\end{equation*}%
which yields that%
\begin{equation*}
(\mathcal{A}\varphi ,\widetilde{\varphi })_{\mathcal{H}}=-\left( B\varphi ,%
\widetilde{\varphi }\right) _{\mathcal{H}}
\end{equation*}%
or%
\begin{equation*}
((\mathcal{A}+B)\varphi ,\widetilde{\varphi })_{\mathcal{H}}=0
\end{equation*}%
for every $\varphi =\left[ p_{0},u_{0},w_{1},w_{2}\right] \in H_{N}^{\bot }.$
Hence, $(\mathcal{A}+B)$ is $H_{N}^{\bot }-$invariant.
\end{proof}

\section{Wellposedness}

This section is devoted to showing the semigroup wellposedness of the PDE
system (\ref{1})-(\ref{3}). The main result of this paper is given as
follows:

\begin{theorem}
\label{wp} Let Condition \ref{cond} hold. Moreover, let $\left\Vert \mathbf{U%
}\right\Vert _{\ast }$ be sufficiently small. Then the operator $(\mathcal{A}%
+B):D(\mathcal{A}+B)\cap H_{N}^{\bot }\rightarrow H_{N}^{\bot }$, as defined
via (\ref{AAA}) and (\ref{feedbackB}), generates a strongly continuous
semigroup $\{e^{(\mathcal{A}+B)t}\}_{t\geq 0}$ on $H_{N}^{\bot }.$ Hence,
for every initial data $\left[ p_{0},{u}_{0},w_{1_{0}},w_{2_{0}}\right] \in
H_{N}^{\bot },$ the solution $\left[ p(t),{u}(t),w_{1}(t),w_{2}(t)\right] $
of problem (\ref{1})-(\ref{3}) is given continuously by 
\begin{equation}
\left[ 
\begin{array}{c}
p(t) \\ 
u(t) \\ 
w_{1}(t) \\ 
w_{2}(t)%
\end{array}%
\right] =e^{(\mathcal{A}+B)t}\left[ 
\begin{array}{c}
p_{0} \\ 
u_{0} \\ 
w_{1_{0}} \\ 
w_{2_{0}}%
\end{array}%
\right] \in C([0,T];H_{N}^{\bot }).
\end{equation}
Moreover, this semigroup is uniformly bounded in time with respect to the
standard $\mathcal{H}$-inner product. (With respect to the special norm in (%
\ref{specnorm}), the semigroup is in fact a contraction.)
\end{theorem}

\begin{remark}
In point of fact, for ambient field $\mathbf{U}$ smooth enough, the operator 
$(\mathcal{A}+B)$ generates a continuous semigroup in the entire phase space 
$\mathcal{H}$. This conclusion can be straightforwardly obtained by invoking
the machinery of \cite{material}. However, this wellposedness on all of $%
\mathcal{H}$ has its downsides: \textbf{(i)} The ambient field requires the
stronger regularity $\mathbf{H}^{3}(\mathcal{O})$ \textbf{(ii)} the
argumentation in \cite{agw, material}, which partly involves linear
perturbation theory, will culminate in the semigroup of $(\mathcal{A}+B)$
not having a uniform bound; in fact the semigroup estimate on all of $%
\mathcal{H}$ will be of exponential order. 
\end{remark}

To prove Theorem \ref{wp}, we will appeal to Lumer-Phillips Theorem that
requires the analysis of the dissipativity and maximality properties of the
semigroup generator $(\mathcal{A}+B)$. We start with the dissipativity for
which our main tool will be the use of the inner product defined in (\ref%
{innpro}):

\subsection{Dissipativity of the Generator $(\mathcal{A}+B)$}

We show the dissipativity property of the generator operator $(\mathcal{A}%
+B) $ in the following lemma:

\begin{lemma}
\label{diss} With reference to problem (\ref{1})-(\ref{3}), the semigroup
generator $(\mathcal{A}+B):D(\mathcal{A}+B)\cap H_{N}^{\bot }\subset
H_{N}^{\bot }\rightarrow H_{N}^{\bot }$ is dissipative with respect to inner
product $((\cdot ,\cdot ))_{H_{N}^{\bot }}$ for $\left\Vert \mathbf{U}%
\right\Vert _{\ast }$ (defined in (\ref{normU})) small enough. In
particular, for $\varphi =\left[ p_{0},u_{0},w_{1},w_{2}\right] \in D(%
\mathcal{A}+B)\cap H_{N}^{\bot },$ 
\begin{equation}
\text{Re}(([\mathcal{A}+B]\varphi ,\varphi ))_{H_{N}^{\bot }}\leq -\frac{%
(\sigma (u_{0}),\epsilon (u_{0}))_{\mathcal{O}}}{4}-\frac{\eta \left\Vert
u_{0}\right\Vert _{\mathcal{O}}^{2}}{4}-\frac{\xi \left\Vert
p_{0}\right\Vert _{\mathcal{O}}^{2}}{2}-\frac{\xi \left\Vert \Delta
w_{1}\right\Vert _{\Omega }^{2}}{2},  \label{dissest}
\end{equation}%
where $\xi $ is specified in (\ref{33}).
\end{lemma}

\begin{proof}
Given $\varphi =\left[ p_{0},u_{0},w_{1},w_{2}\right] \in D(\mathcal{A}%
+B)\cap H_{N}^{\bot },$ we have 
\begin{equation*}
(([\mathcal{A}+B]\varphi ,\varphi ))_{H_{N}^{\bot }}=(-\mathbf{U}\nabla
p_{0}-\text{div}(u_{0})-\text{div}(\mathbf{U})p_{0},p_{0})_{\mathcal{O}}
\end{equation*}%
\begin{equation*}
+(-\nabla p_{0}+\text{div}\sigma (u_{0})-\eta u_{0}-\mathbf{U}\nabla
u_{0},u_{0}-\alpha D(g\cdot \nabla w_{1})e_{3})_{\mathcal{O}}
\end{equation*}%
\begin{equation*}
+(-\nabla p_{0}+\text{div}\sigma (u_{0})-\eta u_{0}-\mathbf{U}\nabla
u_{0},\xi \nabla \psi (p_{0},w_{1}))_{\mathcal{O}}
\end{equation*}%
\begin{equation*}
-\alpha (D(g\cdot \nabla \lbrack w_{2}+\mathbf{U}\nabla
w_{1}])e_{3},u_{0}-\alpha D(g\cdot \nabla w_{1})e_{3}+\xi \nabla \psi
(p_{0},w_{1}))_{\mathcal{O}}
\end{equation*}%
\begin{equation*}
+\xi (\nabla \psi (-\mathbf{U}\nabla p_{0}-\text{div}(u_{0})-\text{div}(%
\mathbf{U})p_{0},w_{2}+\mathbf{U}\nabla w_{1}),u_{0}-\alpha D(g\cdot \nabla
w_{1})e_{3})_{\mathcal{O}}
\end{equation*}%
\begin{equation*}
+\xi ^{2}(\nabla \psi (-\mathbf{U}\nabla p_{0}-\text{div}(u_{0})-\text{div}(%
\mathbf{U})p_{0},w_{2}+\mathbf{U}\nabla w_{1}),\nabla \psi (p_{0},w_{1}))_{%
\mathcal{O}}
\end{equation*}%
\begin{equation*}
+(\Delta w_{2},\Delta w_{1})_{\Omega }+(\Delta (\mathbf{U}\nabla
w_{1}),\Delta w_{1})_{\Omega }
\end{equation*}%
\begin{equation*}
+(p_{0}|_{\Omega }-\left[ 2\nu \partial _{x_{3}}(u_{0})_{3}+\lambda \text{div%
}(u_{0})\right] |_{\Omega },w_{2}+h_{\alpha }\cdot \nabla w_{1}+\xi
w_{1})_{\Omega }
\end{equation*}%
\begin{equation*}
+(h_{\alpha }\cdot \nabla \lbrack w_{2}+\mathbf{U}\nabla
w_{1}],w_{2}+h_{\alpha }\cdot \nabla w_{1}+\xi w_{1})_{\Omega }
\end{equation*}%
\begin{equation*}
-(\Delta ^{2}w_{1},w_{2}+h_{\alpha }\cdot \nabla w_{1}+\xi w_{1})_{\Omega }
\end{equation*}%
\begin{equation*}
+\xi (w_{2}+\mathbf{U}\nabla w_{1},w_{2}+h_{\alpha }\cdot \nabla w_{1}+\xi
w_{1})_{\Omega }.
\end{equation*}%
After integration by parts we then arrive at%
\begin{equation*}
(([\mathcal{A}+B]\varphi ,\varphi ))_{H_{N}^{\bot }}=-(\sigma
(u_{0}),\epsilon (u_{0}))_{\mathcal{O}}-\eta \left\Vert u_{0}\right\Vert _{%
\mathcal{O}}^{2}+\frac{1}{2}\int\limits_{\mathcal{O}}\text{div}(\mathbf{U}%
)[|u_{0}|^{2}-|p_{0}|^{2}]d\mathcal{O}
\end{equation*}%
\begin{equation*}
+2i\text{Im}[(p_{0},\text{div}(u_{0}))_{\mathcal{O}}+(\Delta w_{2},\Delta
w_{1})_{\Omega }]-i\text{Im}[(\mathbf{U}\nabla p_{0},p_{0})_{\mathcal{O}}+(%
\mathbf{U}\nabla u_{0},u_{0})_{\mathcal{O}}]
\end{equation*}%
\begin{equation}
+\sum\limits_{j=1}^{8}I_{j},  \label{est}
\end{equation}%
where above the $I_{j}$ are given by:%
\begin{equation*}
I_{1}=(\nabla p_{0}-\text{div}\sigma (u_{0})+\eta u_{0}+\mathbf{U}\nabla
u_{0},\alpha D(g\cdot \nabla w_{1})e_{3})_{\mathcal{O}}
\end{equation*}%
\begin{equation}
-\alpha (p_{0}|_{\Omega }-\left[ 2\nu \partial _{x_{3}}(u_{0})_{3}+\lambda 
\text{div}(u_{0})\right] |_{\Omega },g\cdot \nabla w_{1})_{\Omega },
\label{I1}
\end{equation}%
\begin{equation*}
I_{2}=(-\nabla p_{0}+\text{div}\sigma (u_{0})-\eta u_{0}-\mathbf{U}\nabla
u_{0},\xi \nabla \psi (p_{0},w_{1}))_{\mathcal{O}}-\xi (\Delta
^{2}w_{1},w_{1})_{\Omega }
\end{equation*}%
\begin{equation}
+(p_{0}|_{\Omega }-\left[ 2\nu \partial _{x_{3}}(u_{0})_{3}+\lambda \text{div%
}(u_{0})\right] |_{\Omega },\xi w_{1})_{\Omega },  \label{I2}
\end{equation}%
\begin{equation}
I_{3}=-\alpha (D(g\cdot \nabla \lbrack w_{2}+\mathbf{U}\nabla
w_{1}])e_{3},u_{0}-\alpha D(g\cdot \nabla w_{1})e_{3}+\xi \nabla \psi
(p_{0},w_{1}))_{\mathcal{O}},  \label{I3}
\end{equation}%
\begin{equation}
I_{4}=\xi (\nabla \psi (-\mathbf{U}\nabla p_{0}-\text{div}(u_{0})-\text{div}(%
\mathbf{U})p_{0},w_{2}+\mathbf{U}\nabla w_{1}),u_{0}-\alpha D(g\cdot \nabla
w_{1})e_{3})_{\mathcal{O}},  \label{I4}
\end{equation}%
\begin{equation}
I_{5}=\xi ^{2}(\nabla \psi (-\mathbf{U}\nabla p_{0}-\text{div}(u_{0})-\text{%
div}(\mathbf{U})p_{0},w_{2}+\mathbf{U}\nabla w_{1}),\nabla \psi
(p_{0},w_{1}))_{\mathcal{O}},  \label{I5}
\end{equation}%
\begin{equation}
I_{6}=(\Delta (\mathbf{U}\nabla w_{1}),\Delta w_{1})_{\Omega }-(\Delta
^{2}w_{1},h_{\alpha }\cdot \nabla w_{1})_{\Omega },  \label{I6}
\end{equation}%
\begin{equation}
I_{7}=(h_{\alpha }\cdot \nabla \lbrack w_{2}+\mathbf{U}\nabla
w_{1}],w_{2})_{\Omega },  \label{I7}
\end{equation}%
\begin{equation*}
I_{8}=(h_{\alpha }\cdot \nabla \lbrack w_{2}+\mathbf{U}\nabla
w_{1}],h_{\alpha }\cdot \nabla w_{1}+\xi w_{1})_{\Omega }
\end{equation*}%
\begin{equation}
+\xi (w_{2}+\mathbf{U}\nabla w_{1},w_{2}+h_{\alpha }\cdot \nabla w_{1}+\xi
w_{1})_{\Omega }.  \label{I8}
\end{equation}%
where we also recall the definition $h_{\alpha}=\mathbf{U}|_{\Omega}-\alpha
g.$ In the course of estimating the terms (\ref{I1})-(\ref{I8}) above, we
will invoke the polynomial 
\begin{equation}
r(a)=a+a^{2}+a^{3}.  \label{poly}
\end{equation}%
We start with $I_{1};$ integrating by parts, we have%
\begin{equation*}
I_{1}=-\alpha (p_{0},\text{div}[D(g\cdot \nabla w_{1})e_{3}])_{\mathcal{O}%
}+\alpha (\sigma (u_{0}),\epsilon (D(g\cdot \nabla w_{1})e_{3})_{\mathcal{O}}
\end{equation*}%
\begin{equation}
+\alpha \eta (u_{0},D(g\cdot \nabla w_{1})e_{3})_{\mathcal{O}}+\alpha (%
\mathbf{U}\nabla u_{0},D(g\cdot \nabla w_{1})e_{3})_{\mathcal{O}}
\label{I1-1}
\end{equation}%
Using the fact that Dirichlet map $D\in L(H_{0}^{\frac{1}{2}+\epsilon
}(\Omega ),H^{1}(\mathcal{O}))$, we have%
\begin{equation}
I_{1}\leq r(\left\Vert \mathbf{U}\right\Vert _{\ast })C\left\{ \left\Vert
u_{0}\right\Vert _{H^{1}(\mathcal{O})}^{2}+\left\Vert p_{0}\right\Vert _{%
\mathcal{O}}^{2}+\left\Vert \Delta w_{1}\right\Vert _{\Omega }^{2}\right\}
\label{A}
\end{equation}%
%
%
%
We continue with $I_{2};$ using the definition of the map $\psi (\cdot
,\cdot )$ in (\ref{Igormap}) and integrating by parts we get%
\begin{equation*}
I_{2}=-\xi \int\limits_{\mathcal{O}}\left\vert p_{0}\right\vert ^{2}d%
\mathcal{O}-\xi (\sigma (u_{0}),\epsilon (\nabla \psi (p_{0},w_{1})))_{%
\mathcal{O}}
\end{equation*}%
\begin{equation*}
+\xi \left\langle \sigma (u_{0})n-p_{0}n,(\nabla \psi
(p_{0},w_{1}),n)n\right\rangle _{\partial \mathcal{O}}-\eta (u_{0},\xi
\nabla \psi (p_{0},w_{1}))_{\mathcal{O}}
\end{equation*}%
\begin{equation*}
(-\mathbf{U}\nabla u_{0},\xi \nabla \psi (p_{0},w_{1}))_{\mathcal{O}%
}-(\Delta ^{2}w_{1},\xi w_{1})_{\Omega }
\end{equation*}%
\begin{equation*}
+(p_{0}|_{\Omega }-\left[ 2\nu \partial _{x_{3}}(u_{0})_{3}+\lambda \text{div%
}(u_{0})\right] |_{\Omega },\xi w_{1})_{\Omega },
\end{equation*}%
whence we obtain%
\begin{equation*}
I_{2}\leq -\xi \left\Vert p_{0}\right\Vert _{\mathcal{O}}^{2}-\xi \left\Vert
\Delta w_{1}\right\Vert _{\Omega }^{2}+\xi r(\left\Vert \mathbf{U}%
\right\Vert _{\ast })C\left\{ \left\Vert u_{0}\right\Vert _{H^{1}(\mathcal{O}%
)}^{2}+\left\Vert p_{0}\right\Vert _{\mathcal{O}}^{2}+\left\Vert \Delta
w_{1}\right\Vert _{\Omega }^{2}\right\}
\end{equation*}%
\begin{equation}
+\xi C\left\{ \left\Vert u_{0}\right\Vert _{H^{1}(\mathcal{O})}\left[
\left\Vert p_{0}\right\Vert _{\mathcal{O}}+\left\Vert \Delta
w_{1}\right\Vert _{\Omega }\right] \right\} .  \label{B}
\end{equation}%
For $I_{3}:$ recalling the boundary condition 
\begin{equation*}
(u_{0})_{3}|_{\Omega }=w_{2}+\mathbf{U}\nabla w_{1},
\end{equation*}%
making use of Lemma 6.1 of \cite{material} and considering the assumptions
made on the geometry in Condition \ref{cond}, we have%
\begin{equation*}
I_{3}\leq \alpha C\left\Vert g\cdot \nabla (u_{0})_{3}\right\Vert _{H^{-%
\frac{1}{2}}(\Omega )}\left\Vert u_{0}-\alpha D(g\cdot \nabla
w_{1})e_{3}+\xi \nabla \psi (p_{0},w_{1})\right\Vert _{\mathcal{O}}
\end{equation*}%
\begin{equation}
\leq C\left[ r(\left\Vert \mathbf{U}\right\Vert _{\ast })\left\{ \left\Vert
u_{0}\right\Vert _{H^{1}(\mathcal{O})}^{2}+\left\Vert \Delta
w_{1}\right\Vert _{\Omega }^{2}\right\} +\xi ^{2}\left\{ \left\Vert
p_{0}\right\Vert _{\mathcal{O}}^{2}+\left\Vert \Delta w_{1}\right\Vert
_{\Omega }^{2}\right\} \right]  \label{C}
\end{equation}%
where we have also implicitly used the Sobolev Embedding Theorem. To
continue with $I_{4}:$%
\begin{equation*}
I_{4}=\xi (\nabla \psi (-\mathbf{U}\nabla p_{0}-\text{div}(\mathbf{U}%
)p_{0},0),u_{0}-\alpha D(g\cdot \nabla w_{1})e_{3})_{\mathcal{O}}
\end{equation*}%
\begin{equation*}
+\xi (\nabla \psi (-\text{div}(u_{0}),u_{0}\cdot \mathbf{n}),u_{0}-\alpha
D(g\cdot \nabla w_{1})e_{3})_{\mathcal{O}}
\end{equation*}%
\begin{equation}
=I_{4a}+I_{4b}  \label{D1}
\end{equation}%
Since $\mathbf{U}\cdot \mathbf{n|}_{\partial \mathcal{O}}\mathbf{=0,}$ we
have that $(\mathbf{U}\nabla p_{0}+$div$(\mathbf{U})p_{0})\in \lbrack H^{1}(%
\mathcal{O})]^{^{\prime }}$ with%
\begin{equation}
\left\Vert \mathbf{U}\nabla p_{0}+\text{div}(\mathbf{U})p_{0}\right\Vert
_{[H^{1}(\mathcal{O})]^{^{\prime }}}\leq C\left\Vert \mathbf{U}\right\Vert
_{\ast }\left\Vert p_{0}\right\Vert _{\mathcal{O}}.  \label{D*}
\end{equation}%
By Lax-Milgram Theorem, we then have%
\begin{equation*}
I_{4a}\leq C\xi \left\Vert \nabla \psi (-\mathbf{U}\nabla p_{0}-\text{div}(%
\mathbf{U})p_{0},0)\right\Vert _{\mathcal{O}}\left\Vert u_{0}-\alpha
D(g\cdot \nabla w_{1})e_{3}\right\Vert _{\mathcal{O}}
\end{equation*}%
\begin{equation}
\leq C\xi r(\left\Vert \mathbf{U}\right\Vert _{\ast })\left\{ \left\Vert
u_{0}\right\Vert _{H^{1}(\mathcal{O})}^{2}+\left\Vert p_{0}\right\Vert _{%
\mathcal{O}}^{2}+\left\Vert \Delta w_{1}\right\Vert _{\Omega }^{2}\right\}
\label{D2}
\end{equation}%
and similarly%
\begin{equation}
I_{4b}\leq C\xi r(\left\Vert \mathbf{U}\right\Vert _{\ast })\left\{
\left\Vert u_{0}\right\Vert _{H^{1}(\mathcal{O})}^{2}+\left\Vert \Delta
w_{1}\right\Vert _{\Omega }^{2}\right\} .  \label{D3}
\end{equation}%
Now, applying (\ref{D2})-(\ref{D3}) to (\ref{D1}) gives%
\begin{equation}
I_{4}\leq C\xi r(\left\Vert \mathbf{U}\right\Vert _{\ast })\left\{
\left\Vert u_{0}\right\Vert _{H^{1}(\mathcal{O})}^{2}+\left\Vert
p_{0}\right\Vert _{\mathcal{O}}^{2}+\left\Vert \Delta w_{1}\right\Vert
_{\Omega }^{2}\right\} .  \label{D}
\end{equation}%
Estimating $I_{5}:$ we proceed as before done for $I_{4}$ and invoke (\ref%
{D*}), Lax Milgram Theorem and the estimate (\ref{psireg}) to have%
\begin{equation}
I_{5}\leq C\xi ^{2}\left[ \left\Vert \mathbf{U}\right\Vert _{\ast }\left\{
\left\Vert p_{0}\right\Vert _{\mathcal{O}}^{2}+\left\Vert \Delta
w_{1}\right\Vert _{\Omega }^{2}\right\} +\left\Vert u_{0}\right\Vert _{H^{1}(%
\mathcal{O})}^{2}\right]  \label{E}
\end{equation}%
For $I_{6},$ in order to estimate the second term in (\ref{I6}), we follow
the standard calculations used for the flux multipliers and the commutator
symbol given by%
\begin{equation}
\lbrack P,Q]f=P(Qf)-Q(Pf)  \label{F*}
\end{equation}%
for the differential operators $P$ and $Q$. Hence,%
\begin{align}
-(\Delta ^{2}w_{1},h_{\alpha}\cdot \nabla w_{1})_{\Omega }=& ~(\nabla \Delta
w_{1},\nabla (h_{\alpha}\cdot \nabla w_{1}))_{\Omega } \\
=& ~-(\Delta w_{1},\Delta (h_{\alpha}\cdot \nabla w_{1}))_{\Omega
}+\int_{\partial \Omega }(h_{\alpha}\cdot \mathbf{\nu })|\Delta
w_{1}|^{2}d\partial \Omega ,
\end{align}%
where, in the first identity we have directly invoked the clamped plate
boundary conditions, and in the second we have used the fact that $%
w_{1}=\partial _{\mathbf{\nu }}w_{1}=0$ on $\partial \Omega $ which yields
that 
\begin{equation*}
\frac{\partial }{\partial \mathbf{\nu }}(h_{\alpha}\cdot \nabla
w_{1})=(h_{\alpha}\cdot \mathbf{\nu })\frac{\partial ^{2}w_{1}}{\partial 
\mathbf{\nu }}=(h_{\alpha}\cdot \mathbf{\nu })(\Delta w_{1}\big|_{\partial
\Omega }).
\end{equation*}
\noindent (See \cite{lagnese} or \cite[p.305]{LT}). Using the commutator
bracket $[\cdot ,\cdot ]$, we can rewrite the latter relation as 
\begin{equation*}
-(\Delta ^{2}w_{1},h_{\alpha}\cdot \nabla w_{1})_{\Omega }=~-(\Delta
w_{1},[\Delta ,h_{\alpha}\cdot \nabla ]w_{1})_{\Omega }-(\Delta
w_{1},h_{\alpha}\cdot \nabla (\Delta w_{1}))_{\Omega }+\int_{\partial \Omega
}(h_{\alpha}\cdot \mathbf{\nu })|\Delta w_{1}|^{2}d\partial \Omega .
\end{equation*}%
With Green's relations once more: 
\begin{align}
-(\Delta ^{2}w_{1},h_{\alpha}\cdot \nabla w_{1})_{\Omega }=& ~-(\Delta
w_{1},[\Delta ,h_{\alpha}\cdot \nabla ]w_{1})_{\Omega }-\frac{1}{2}%
\int_{\partial \Omega }(h_{\alpha}\cdot \mathbf{\nu })|\Delta
w_{1}|^{2}d\partial \Omega  \notag \\
& +\frac{1}{2}\int_{\Omega }\big[\text{div}(h_{\alpha})\big]|\Delta
w_{1}|^{2}d\Omega -i\text{Im}(\Delta w_{1},h_{\alpha}\cdot \nabla (\Delta
w_{1}))_{\Omega }  \notag \\
& +\int_{\partial \Omega }(h_{\alpha}\cdot \mathbf{\nu })|\Delta
w_{1}|^{2}d\partial \Omega .  \label{use1}
\end{align}%
Thus, 
\begin{align}
-(\Delta ^{2}w_{1},h_{\alpha}\cdot \nabla w_{1})_{\Omega }=& ~-(\Delta
w_{1},[\Delta ,h_{\alpha}\cdot \nabla ]w_{1})_{\Omega }+\frac{1}{2}%
\int_{\partial \Omega }(h_{\alpha}\cdot \mathbf{\nu })|\Delta
w_{1}|^{2}d\partial \Omega  \notag \\
& +\frac{1}{2}\int_{\Omega }\big[\text{div}(h_{\alpha})\big]|\Delta
w_{1}|^{2}d\Omega -i\text{Im}(\Delta w_{1},h_{\alpha}\cdot \nabla (\Delta
w_{1})).  \label{use2}
\end{align}%
Since $h_{\alpha}=\mathbf{U}\big|_{\Omega }-\alpha g$, where $g$ is an
extension of $\mathbf{\nu }(\mathbf{x})$, we will have then%
\begin{equation}
-\text{Re}(\Delta ^{2}w_{1},h_{\alpha}\cdot \nabla w_{1})_{\Omega }=~\dfrac{1%
}{2}\int_{\partial \Omega }(\mathbf{U}\cdot \mathbf{\nu }-\alpha )|\Delta
w_{1}|^{2}d\partial \Omega +\dfrac{1}{2}\int_{\Omega }\text{div}%
(h_{\alpha})|\Delta w_{1}|^{2}d\Omega -\text{Re}(\Delta w_{1},[\Delta
,h_{\alpha}\cdot \nabla ]w_{1})_{\Omega }  \label{com1}
\end{equation}%
Since we can explicitly compute the commutator 
\begin{align*}
\lbrack \Delta ,{h_{\alpha}}\cdot \nabla ]w_{1}=& (\Delta h_{1})(\partial
_{x_{1}}w_{1})+2(\partial _{x_{1}}h_{1})(\partial
_{x_{1}}^{2}w_{1})+2(\partial _{x_{2}}h_{2})(\partial
_{x_{2}}^{2}w_{1})+(\Delta h_{2})(\partial _{x_{2}}w_{1}) \\
& +2\text{div}(h_{\alpha})(\partial _{x_{1}}\partial _{x_{2}}w_{1}),
\end{align*}%
and 
\begin{equation}
\big|\big|\lbrack \Delta ,{h_{\alpha}}\cdot \nabla ]w_{1}\big|\big|%
_{L^{2}(\Omega )}\leq Cr(\left\Vert \mathbf{U}\right\Vert _{\ast })||\Delta
w_{1}||_{L^{2}(\Omega )}.  \label{commest}
\end{equation}%
combining (\ref{com1})-(\ref{commest}) we eventually get%
\begin{equation}
-\text{Re}(\Delta ^{2}w_{1},h_{\alpha }\cdot \nabla w_{1})_{\Omega }\leq 
\frac{1}{2}\int\limits_{\partial \Omega }[\mathbf{U\cdot \nu -}\alpha
]\left\vert \Delta w_{1}\right\vert ^{2}d\partial \Omega +Cr(\left\Vert 
\mathbf{U}\right\Vert _{\ast })\left\Vert \Delta w_{1}\right\Vert _{\Omega
}^{2}.  \label{F1}
\end{equation}%
Moreover, for the first term of (\ref{I6}), we have%
\begin{equation*}
(\Delta (\mathbf{U}\nabla w_{1}),\Delta w_{1})_{\Omega }=(\mathbf{U}\nabla
w_{1}),\Delta w_{1})_{\Omega }-([\mathbf{U\cdot }\nabla ,\Delta
]w_{1},\Delta w_{1})_{\Omega }
\end{equation*}%
\begin{equation*}
=\int\limits_{\partial \Omega }(\mathbf{U\cdot \nu })\left\vert \Delta
w_{1}\right\vert ^{2}d\partial \Omega -\int\limits_{\partial \Omega }\text{%
div}(\mathbf{U)}\left\vert \Delta w_{1}\right\vert ^{2}d\partial \Omega
\end{equation*}%
\begin{equation*}
-([\mathbf{U\cdot }\nabla ,\Delta ]w_{1},\Delta w_{1})_{\Omega
}-\int\limits_{\Omega }\Delta w_{1}\mathbf{U\cdot }\nabla (\Delta
w_{1})d\Omega
\end{equation*}%
where we also use the commutator expression in (\ref{F*}). This gives us%
\begin{equation}
\text{Re}(\Delta (\mathbf{U}\nabla w_{1}),\Delta w_{1})_{\Omega }\leq \frac{1%
}{2}\int\limits_{\partial \Omega }(\mathbf{U\cdot \nu })\left\vert \Delta
w_{1}\right\vert ^{2}d\partial \Omega +Cr(\left\Vert \mathbf{U}\right\Vert
_{\ast })\left\Vert \Delta w_{1}\right\Vert _{\Omega }^{2}.  \label{F2}
\end{equation}%
Now applying (\ref{F1})-(\ref{F2}) to (\ref{I6}), we obtain%
\begin{equation}
\text{Re}I_{6}\leq \int\limits_{\partial \Omega }[\mathbf{U\cdot \nu -}\frac{%
\alpha }{2}]\left\vert \Delta w_{1}\right\vert ^{2}d\partial \Omega
+Cr(\left\Vert \mathbf{U}\right\Vert _{\ast })\left\Vert \Delta
w_{1}\right\Vert _{\Omega }^{2}.  \tag{F}
\end{equation}%
To estimate $I_{7}:$ since $w_{2}\in H_{0}^{1}(\Omega ),$ we have 
\begin{equation*}
\text{Re}(h_{\alpha }\cdot \nabla w_{2},w_{2})_{\Omega }=-\frac{1}{2}%
\int\limits_{\Omega }\text{div}(h_{\alpha }\mathbf{)}\left\vert
w_{2}\right\vert ^{2}d\Omega
\end{equation*}%
\begin{equation*}
=-\frac{1}{2}\int\limits_{\Omega }\text{div}(h_{\alpha }\mathbf{)}\left\vert
(u_{0})_{3}-\mathbf{U}\nabla w_{1}\right\vert ^{2}d\Omega
\end{equation*}%
after using the boundary condition in $\mathbf{(A.v)}.$ Applying the last
relation to RHS of (\ref{I7}) and recalling that $h_{\alpha }=\mathbf{U|}%
_{\Omega }-\alpha g,$ we get%
\begin{equation*}
\text{Re}I_{7}=\text{Re}(h_{\alpha }\cdot \nabla w_{2},w_{2})_{\Omega }+%
\text{Re}(h_{\alpha }\cdot \nabla (\mathbf{U}\nabla w_{1}),(u_{0})_{3}-%
\mathbf{U}\nabla w_{1})_{\mathcal{O}}
\end{equation*}%
\begin{equation}
\leq Cr(\left\Vert \mathbf{U}\right\Vert _{\ast })\left\{ \left\Vert
u_{0}\right\Vert _{H^{1}(\mathcal{O})}^{2}+\left\Vert \Delta
w_{1}\right\Vert _{\Omega }^{2}\right\}  \label{G}
\end{equation}%
where we also implicitly use Sobolev Trace Theorem. Lastly, for the term $%
I_{8}$, we proceed in a manner similar to that adopted for $I_{7}$ and we
have%
\begin{equation*}
I_{8}=(h_{\alpha }\cdot \nabla (u_{0})_{3},h_{\alpha }\cdot \nabla w_{1}+\xi
w_{1})_{\Omega }
\end{equation*}%
\begin{equation*}
+\xi ((u_{0})_{3},(u_{0})_{3}-\mathbf{U}\cdot \nabla w_{1}+h_{\alpha }\cdot
\nabla w_{1}+\xi w_{1})_{\Omega }
\end{equation*}%
\begin{equation*}
\leq C\left[ r(\left\Vert \mathbf{U}\right\Vert _{\ast })+\xi ^{2}\right]
\left\{ \left\Vert u_{0}\right\Vert _{H^{1}(\mathcal{O})}^{2}+\left\Vert
\Delta w_{1}\right\Vert _{\Omega }^{2}\right\}
\end{equation*}%
\begin{equation}
+C\xi \left[ \left\Vert u_{0}\right\Vert _{H^{1}(\mathcal{O}%
)}^{2}+r(\left\Vert \mathbf{U}\right\Vert _{\ast })\left\{ \left\Vert
u_{0}\right\Vert _{H^{1}(\mathcal{O})}^{2}+\left\Vert \Delta
w_{1}\right\Vert _{\Omega }^{2}\right\} \right]  \label{H}
\end{equation}%
Now, if we apply (\ref{A})-(\ref{H}) to RHS of (\ref{est}), we obtain%
\begin{equation*}
\text{Re}(([\mathcal{A}+B]\varphi ,\varphi ))_{H_N^{\bot}}\leq -(\sigma
(u_{0}),\epsilon (u_{0}))_{\mathcal{O}}-\eta \left\Vert u_{0}\right\Vert _{%
\mathcal{O}}^{2}-\xi \left\Vert p_{0}\right\Vert _{\mathcal{O}}^{2}-\xi
\left\Vert \Delta w_{1}\right\Vert _{\Omega }^{2}
\end{equation*}%
\begin{equation*}
+\int\limits_{\partial \Omega }[\mathbf{U\cdot \nu -}\frac{\alpha }{2}%
]\left\vert \Delta w_{1}\right\vert ^{2}d\partial \Omega
\end{equation*}%
\begin{equation*}
+C\left[ r_{\mathbf{U}}+\xi r_{\mathbf{U}}+\xi ^{2}+\xi \right] \left\Vert
u_{0}\right\Vert _{H^{1}(\mathcal{O})}^{2}
\end{equation*}%
\begin{equation*}
+C\left[ r_{\mathbf{U}}+\xi r_{\mathbf{U}}+\xi ^{2}+\xi ^{2}r_{\mathbf{U}}%
\right] \left\{ \left\Vert p_{0}\right\Vert _{\mathcal{O}}^{2}+\left\Vert
\Delta w_{1}\right\Vert _{\Omega }^{2}\right\}
\end{equation*}%
\begin{equation}
+C\xi \left\Vert u_{0}\right\Vert _{H^{1}(\mathcal{O})}^{2}\left\{
\left\Vert p_{0}\right\Vert _{\mathcal{O}}+\left\Vert \Delta
w_{1}\right\Vert _{\Omega }\right\}  \label{31}
\end{equation}%
where, for the simplicity, we have set $r_{\mathbf{U}}=$ $r(\left\Vert 
\mathbf{U}\right\Vert _{\ast }).$ We recall now the value of $\alpha
=2\left\Vert \mathbf{U}\right\Vert _{\ast }$ to get%
\begin{equation*}
\text{Re}(([\mathcal{A}+B]\varphi ,\varphi ))_{H_N^{\bot}}\leq -(\sigma
(u_{0}),\epsilon (u_{0}))_{\mathcal{O}}-\eta \left\Vert u_{0}\right\Vert _{%
\mathcal{O}}^{2}-\xi \left\Vert p_{0}\right\Vert _{\mathcal{O}}^{2}-\xi
\left\Vert \Delta w_{1}\right\Vert _{\Omega }^{2}
\end{equation*}%
\begin{equation*}
+\left[ (C_{1}+C_{2}r_{\mathbf{U}})\xi ^{2}+C_{2}r_{\mathbf{U}}\xi +C_{2}r_{%
\mathbf{U}}\right] \left\{ \left\Vert p_{0}\right\Vert _{\mathcal{O}%
}^{2}+\left\Vert \Delta w_{1}\right\Vert _{\Omega }^{2}\right\}
\end{equation*}%
\begin{equation*}
+\frac{1}{2}\left\{ (\sigma (u_{0}),\epsilon (u_{0}))_{\mathcal{O}}+\eta
\left\Vert u_{0}\right\Vert _{\mathcal{O}}^{2}\right\}
\end{equation*}%
\begin{equation}
+C_{3}\left[ r_{\mathbf{U}}+\xi r_{\mathbf{U}}+\xi ^{2}+\xi \right]
\left\Vert u_{0}\right\Vert _{H^{1}(\mathcal{O})}^{2}  \label{32}
\end{equation}%
where the positive constants $C_{1},C_{2}$ and $C_{3}$ are obtained with the
application of Holder-Young and Korn's inequalities and $C_{2}$ depends on
the constant in Korn's inequality. We now specify $\xi $ be a zero of the
equation%
\begin{equation*}
(C_{1}+C_{2}r_{\mathbf{U}})\xi ^{2}+(C_{2}r_{\mathbf{U}}-\frac{1}{2})\xi
+C_{2}r_{\mathbf{U}}=0.
\end{equation*}%
Namely, 
\begin{equation}
\xi =\frac{\frac{1}{2}-C_{2}r_{\mathbf{U}}}{2(C_{1}+C_{2}r_{\mathbf{U}})}-%
\frac{\sqrt{(\frac{1}{2}-C_{2}r_{\mathbf{U}})^{2}-4C_{2}(C_{1}+C_{2}r_{%
\mathbf{U}})r_{\mathbf{U}}}}{2(C_{1}+C_{2}r_{\mathbf{U}})}  \label{33}
\end{equation}%
where the radicand is nonnegative for $\left\Vert \mathbf{U}\right\Vert
_{\ast }$ sufficiently small. Then (\ref{32}) becomes%
\begin{equation*}
\text{Re}(([\mathcal{A}+B]\varphi ,\varphi ))_{H_N^{\bot}}\leq -\frac{%
(\sigma (u_{0}),\epsilon (u_{0}))_{\mathcal{O}}}{4}-\eta \frac{\left\Vert
u_{0}\right\Vert _{\mathcal{O}}^{2}}{4}-\frac{\xi }{2}\left\Vert
p_{0}\right\Vert _{\mathcal{O}}^{2}-\frac{\xi }{2}\left\Vert \Delta
w_{1}\right\Vert _{\Omega }^{2}
\end{equation*}%
\begin{equation*}
-\frac{(\sigma (u_{0}),\epsilon (u_{0}))_{\mathcal{O}}}{4}-\eta \frac{%
\left\Vert u_{0}\right\Vert _{\mathcal{O}}^{2}}{4}
\end{equation*}%
\begin{equation*}
+C_{K}\left[ r_{\mathbf{U}}+\xi r_{\mathbf{U}}+\xi ^{2}+\xi \right] \left\{
(\sigma (u_{0}),\epsilon (u_{0}))_{\mathcal{O}}+\eta \left\Vert
u_{0}\right\Vert _{\mathcal{O}}^{2}\right\} .
\end{equation*}%
With $\xi $ as prescribed in (\ref{33}), we now have the dissipativity
estimate (\ref{dissest}), for $\left\Vert \mathbf{U}\right\Vert _{\ast }$
small enough. (Here we also implicitly re-use Korn's inequality and $C_{K}$
is the constant there). This concludes the proof of Lemma \ref{diss}.
\end{proof}

\subsection{Maximality of the Generator $(\mathcal{A}+B)$}

In order to complete the proof of Theorem \ref{wp}, we also need to show
that the semigroup generator $(\mathcal{A}+B):D(\mathcal{A}+B)\cap
H_{N}^{\bot }\subset H_{N}^{\bot }\rightarrow H_{N}^{\bot }$ is maximal
dissipative. This is given in the following lemma:

\begin{lemma}
\label{md} With reference to problem (\ref{1})-(\ref{3}), the semigroup
generator $(\mathcal{A}+B):D(\mathcal{A}+B)\cap H_{N}^{\bot }\subset
H_{N}^{\bot }\rightarrow H_{N}^{\bot }$ is maximal dissipative. In other
words, the following range condition holds: 
\begin{equation}
\text{Range}[\lambda I-(\mathcal{A}+B)]=H_{N}^{\bot }  \label{range}
\end{equation}%
for some $\lambda >0.$
\end{lemma}

\vspace{0.3cm} \textbf{Proof of Lemma \ref{md}}\newline

Proof of relation (\ref{range}) is based on showing that $[\lambda I-(%
\mathcal{A}+B)]^{-1}\in \mathcal{L}(H_{N}^{\bot }).$ For this, we appeal to
linear operator theory and exploit Lemma \ref{pazy} in Appendix as our main
tool. So, with respect to Lemma \ref{pazy} the requirements to be shown are: 
\newline
\newline
$\mathbf{(M-I)}$ \ \ Range$[\lambda I-(\mathcal{A}+B)]$ is dense in $%
H_{N}^{\bot },$\newline
$\mathbf{(M-II)}$ \ $[\lambda I-(\mathcal{A}+B)]$ is a closed operator.%
\newline
$\mathbf{(M-III)}$ There is an $m>0$ such that 
\begin{equation*}
\left\Vert \left\vert [\lambda I-(\mathcal{A}+B)]\varphi \right\vert
\right\Vert _{H_{N}^{\bot }} \geq m\left\Vert \left\vert \varphi \right\vert
\right\Vert _{H_{N}^{\bot }}
\end{equation*}
for all $\varphi \in D([\lambda I-(\mathcal{A}+B)])\cap {H_{N}^{\bot }}=D(%
\mathcal{A}+B)\cap {H_{N}^{\bot }}.$ \newline\newline

\noindent \underline{\textbf{STEP (M-I):}} Firstly, to prove that Range$[\lambda I-(%
\mathcal{A}+B)]$ is dense in $H_{N}^{\bot },$ we use the fact that 
\begin{equation*}
\text{Range}[\lambda I-(\mathcal{A}+B)]=Null([\lambda I-(\mathcal{A}%
+B)]^{\ast })^{\bot }
\end{equation*}%
which is given in the following lemma:

\begin{lemma}
\label{denserange} Let parameter $\lambda >0$ be given. Then for $\left\Vert 
\mathbf{U}\right\Vert _{\ast }$ sufficiently small,%
\begin{equation*}
Null[\lambda I-(\mathcal{A}+B)^{\ast }]=\{0\}
\end{equation*}
\end{lemma}

\begin{proof}
Suppose that $\varphi =\left[ p_{0},u_{0},w_{1},w_{2}\right] \in D((\mathcal{%
A}+B)^{\ast })\cap H_{N}^{\bot }$ satisfies%
\begin{equation}
[\lambda I-(\mathcal{A}+B)^{\ast }]\varphi =0.  \label{1.1}
\end{equation}%
In PDE terms, this is%
\begin{equation}
\left\{ 
\begin{array}{c}
\lambda p_{0}-\mathbf{U}\nabla p_{0}-\text{div}(u_{0})=0\text{ \ \ in \ \ }%
\mathcal{O} \\ 
\lambda u_{0}-\nabla p_{0}-\text{div}\sigma (u_{0})+\eta u_{0}-\mathbf{U}%
\nabla u_{0}+\text{div}(\mathbf{U})u_{0}=0\text{ \ \ in \ \ }\mathcal{O} \\ 
u_{0}\cdot n=0\text{ \ \ on \ \ }S \\ 
u_{0}\cdot n=w_{2}\text{ \ \ on \ \ }\Omega \\ 
\lambda w_{1}+w_2-{{\mathring{A}}^{-1}\left\{ \text{div}{{[U}_{1},U_{2}]{+%
\mathbf{U\cdot }\nabla }}\right\} }\left[ p_{0}+2\nu \partial
_{x_{3}}(u_{0})_{3}+\lambda \text{div}(u_{0})-\Delta ^{2}w_{1}\right]
_{\Omega } \\ 
-\mathbf{U}\mathbb{\cdot }\nabla w_{1}-\Delta {{\mathring{A}}^{-1}\nabla }%
^{\ast }(\mathbb{\nabla \cdot }(\mathbf{U}\mathbb{\cdot }\nabla w_{1})%
\mathbb{)}=0\text{ \ \ in \ \ }\Omega \\ 
\lambda w_{2}+\left[ p_{0}+2\nu \partial _{x_{3}}(u_{0})_{3}+\lambda \text{%
div}(u_{0})\right] |_{\Omega }-\Delta ^{2}w_{1}=0\text{ \ \ in \ \ }\Omega
\\ 
w_{1}|_{\partial \Omega }=\frac{\partial w_{1}}{\partial \nu }|_{\partial
\Omega }=0%
\end{array}%
\right.  \label{1.2}
\end{equation}%
Since we have from (\ref{1.1}) 
\begin{equation}
0=\lambda \left\Vert \varphi \right\Vert _{\mathcal{H}}^{2}-((\mathcal{A}%
+B)^{\ast }\varphi ,\varphi )_{\mathcal{H}}  \label{1.3}
\end{equation}%
integrating by parts as usual, we get%
\begin{equation*}
\lambda \left\Vert \varphi \right\Vert _{\mathcal{H}}^{2}+(\sigma
(u_{0}),\epsilon (u_{0}))_{\mathcal{O}}+\eta \left\Vert u_{0}\right\Vert _{%
\mathcal{O}}^{2}
\end{equation*}%
\begin{equation*}
=-\frac{1}{2}\int\limits_{\mathcal{O}}\text{div}(\mathbf{U}%
)[|p_{0}|^{2}+3|u_{0}|^{2}]d\mathcal{O}
\end{equation*}%
\begin{equation*}
+\left( {\left\{ \text{div}{{[U}_{1},U_{2}]{+\mathbf{U\cdot }\nabla }}%
\right\} }\left[ p_{0}+2\nu \partial _{x_{3}}(u_{0})_{3}+\lambda \text{div}%
(u_{0})-\Delta ^{2}w_{1}\right] _{\Omega },w_{1}\right) _{\Omega }
\end{equation*}%
\begin{equation}
+\left( \Delta \lbrack \mathbf{U}\mathbb{\cdot }\nabla w_{1}],\Delta
w_{1}\right) _{\Omega }+\left( {\nabla }^{\ast }(\mathbb{\nabla \cdot }(%
\mathbf{U}\mathbb{\cdot }\nabla w_{1})\mathbb{)},\Delta w_{1}\right)
_{\Omega }  \label{1.5}
\end{equation}%
To handle the terms on RHS of (\ref{1.5}), we firstly invoke the map given
in (\ref{Igormap}) and apply the multiplier $\nabla \psi (p_{0},w_{1})$ to
the fluid equation (\ref{1.2})$_{2}.$ This gives%
\begin{equation*}
\lambda \left( u_{0},\nabla \psi (p_{0},w_{1})\right) _{\mathcal{O}}-\left(
\nabla p_{0},\nabla \psi (p_{0},w_{1})\right) _{\mathcal{O}}-\left( \text{div%
}\sigma (u_{0}),\nabla \psi (p_{0},w_{1})\right) _{\mathcal{O}}
\end{equation*}%
\begin{equation}
+\eta \left( u_{0},\nabla \psi (p_{0},w_{1})\right) _{\mathcal{O}}-\left( 
\mathbf{U}\nabla u_{0},\nabla \psi (p_{0},w_{1})\right) _{\mathcal{O}}+\left(%
\text{div}(\mathbf{U}) u_{0},\nabla \psi (p_{0},w_{1})\right) _{\mathcal{O}%
}=0  \label{1.6}
\end{equation}%
Let us look at the terms of (\ref{1.6}):%
\begin{equation*}
-\left( \nabla p_{0},\nabla \psi (p_{0},w_{1})\right) _{\mathcal{O}%
}=\int\limits_{\partial \mathcal{O}}(p_{0}\cdot n)\nabla \psi
(p_{0},w_{1})d\partial \mathcal{O}
\end{equation*}%
\begin{equation*}
+\int\limits_{\mathcal{O}}p_{0}\text{div}(\nabla \psi (p_{0},w_{1}))d%
\mathcal{O}
\end{equation*}%
\begin{equation}
=-\int\limits_{\mathcal{O}}|p_{0}|^{2}d\mathcal{O-}\int\limits_{%
\Omega}p_{0}w_{1}d\Omega.  \label{i}
\end{equation}%
Also, 
\begin{equation*}
-\left( \text{div}\sigma (u_{0}),\nabla \psi (p_{0},w_{1})\right) _{\mathcal{%
O}}+\eta \left( u_{0},\nabla \psi (p_{0},w_{1})\right) _{\mathcal{O}}
\end{equation*}%
\begin{equation*}
=\left( \sigma (u_{0}),\epsilon (\nabla \psi (p_{0},w_{1}))\right) _{%
\mathcal{O}}-\left\langle \sigma (u_{0})\cdot n,\nabla \psi
(p_{0},w_{1})\right\rangle _{\partial \mathcal{O}}
\end{equation*}%
\begin{equation}
+\eta \left( u_{0},\nabla \psi (p_{0},w_{1})\right) _{\mathcal{O}}
\label{ii}
\end{equation}%
Applying (\ref{i})-(\ref{ii}) to (\ref{1.6}), we then have%
\begin{equation*}
\int\limits_{\mathcal{O}}|p_{0}|^{2}d\mathcal{O=}\lambda \left( u_{0},\nabla
\psi (p_{0},w_{1})\right) _{\mathcal{O}}-\left( \mathbf{U}\nabla
u_{0},\nabla \psi (p_{0},w_{1})\right) _{\mathcal{O}}
\end{equation*}%
\begin{equation*}
+\left(\text{div}(\mathbf{U}) u_{0},\nabla \psi (p_{0},w_{1})\right) _{%
\mathcal{O}}-\left( \left[ p_{0}+2\nu \partial _{x_{3}}(u_{0})_{3}+\lambda 
\text{div}(u_{0})\right] _{\Omega },w_{1}\right) _{\Omega }
\end{equation*}%
\begin{equation}
+\left( \sigma (u_{0}),\epsilon (\nabla \psi (p_{0},w_{1}))\right) _{%
\mathcal{O}}+\eta \left( u_{0},\nabla \psi (p_{0},w_{1})\right) _{\mathcal{O}%
}  \label{1.7}
\end{equation}%
Subsequently, we apply the multiplier $w_{1}$ to the structural equation in (%
$\ref{1.2})_7$, and use (\ref{1.7}) to get%
\begin{equation*}
\int\limits_{\mathcal{O}}|p_{0}|^{2}d\mathcal{O+}\left( \Delta
^{2}w_{1},w_{1}\right) _{\Omega }\mathcal{=\lambda (}w_{2},w_{1}\mathcal{)}%
_{\Omega }+\lambda \left( u_{0},\nabla \psi (p_{0},w_{1})\right) _{\mathcal{O%
}}
\end{equation*}%
\begin{equation*}
+\left( \sigma (u_{0}),\epsilon (\nabla \psi (p_{0},w_{1}))\right) _{%
\mathcal{O}}+\eta \left( u_{0},\nabla \psi (p_{0},w_{1})\right) _{\mathcal{O}%
}
\end{equation*}
\begin{equation}
-\left( \mathbf{U}\nabla u_{0},\nabla \psi (p_{0},w_{1})\right) _{\mathcal{O}%
} +\left(\text{div}(\mathbf{U}) u_{0},\nabla \psi (p_{0},w_{1})\right) _{%
\mathcal{O}}  \label{1.8}
\end{equation}
To estimate the terms on RHS of (\ref{1.8}), we appeal to the elliptic
regularity results for solutions of second order BVPs on corner domains \cite%
{Dauge_1}. At this point, using the geometrical assumptions in Condition \ref%
{cond} and the higher regularity estimate 
\begin{equation*}
\left\Vert \psi (p,w)\right\Vert _{H^{2}(\mathcal{O)}}\leq C\left[
\left\Vert p\right\Vert _{\mathcal{O}}+\left\Vert w_{ext}\right\Vert _{H^{%
\frac{1}{2}+\varepsilon }(\partial \mathcal{O)}}\right]
\end{equation*}%
\begin{equation}
\leq C[\left\Vert p\right\Vert _{\mathcal{O}}+\left\Vert w\right\Vert
_{H_{0}^{2}(\Omega \mathcal{)}}],  \label{1.8.5}
\end{equation}%
where 
\begin{equation*}
w_{ext}(x)=\left\{ 
\begin{array}{c}
0,\text{ \ \ }x\in S \\ 
w(x),\text{ \ \ }x\in \Omega%
\end{array}%
\right.
\end{equation*}%
we obtain 
\begin{equation}
\int\limits_{\mathcal{O}}|p_{0}|^{2}d\mathcal{O+}\int\limits_{\Omega
}|\Delta w_{1}|^{2}d\Omega \leq C_{\epsilon }r(\left\Vert \mathbf{U}%
\right\Vert _{\ast })\left\{ \sigma (u_{0}),\epsilon (u_{0}))_{\mathcal{O}%
}+\eta \left\Vert u_{0}\right\Vert _{\mathcal{O}}^{2}+\lambda \left\Vert
\varphi \right\Vert _{\mathcal{H}}^{2}\right\}  \label{1.9}
\end{equation}%
Here, we also used Holder-Young Inequalities and $r(\cdot )$ and $\left\Vert 
\mathbf{U}\right\Vert _{\ast }$ are given as in (\ref{poly}) and (\ref{normU}%
), respectively. Now, to proceed with the second term on RHS of (\ref{1.5}):%
\begin{equation*}
\left( {\left\{ \text{div}{{[U}_{1},U_{2}]{+\mathbf{U\cdot }\nabla }}%
\right\} }\left[ p_{0}+2\nu \partial _{x_{3}}(u_{0})_{3}+\lambda \text{div}%
(u_{0})-\Delta ^{2}w_{1}\right] _{\Omega },w_{1}\right) _{\Omega }
\end{equation*}%
\begin{equation*}
=\left( {\left\{ \text{div}{{[U}_{1},U_{2}]{+\mathbf{U\cdot }\nabla }}%
\right\} }\left[ p_{0}+2\nu \partial _{x_{3}}(u_{0})_{3}+\lambda \text{div}%
(u_{0})\right] _{\Omega },w_{1}\right) _{\Omega }
\end{equation*}%
\begin{equation*}
-\left( {\left\{ \text{div}{{[U}_{1},U_{2}]{+\mathbf{U\cdot }\nabla }}%
\right\} }\Delta ^{2}w_{1},w_{1}\right) _{\Omega }
\end{equation*}%
\begin{equation}
=K_{1}+K_{2}  \label{1.9.5}
\end{equation}%
For $K_{1}:$%
\begin{equation*}
K_{1}=\left( {\left\{ \text{div}{{[U}_{1},U_{2}]{+\mathbf{U\cdot }\nabla }}%
\right\} }\left[ p_{0}+2\nu \partial _{x_{3}}(u_{0})_{3}+\lambda \text{div}%
(u_{0})\right] _{\Omega },w_{1}\right) _{\Omega }
\end{equation*}%
\begin{equation}
=-\left( \left[ p_{0}+2\nu \partial _{x_{3}}(u_{0})_{3}+\lambda \text{div}%
(u_{0})\right] _{\Omega },{{\mathbf{U\cdot }\nabla }}w_{1}\right) _{\Omega }
\label{1.10}
\end{equation}%
To handle the term on RHS of (\ref{1.10}): Let $D_{\Omega }:H_{0}^{\frac{1}{2%
}+\epsilon }(\Omega )\rightarrow H^{1}(\mathcal{O})$ be defined by 
\begin{equation}
D_{\Omega }g=f\Leftrightarrow \left\{ 
\begin{array}{c}
-\Delta f=0\text{ \ \ \ in \ }\mathcal{O} \\ 
f|_{S}=0\text{ \ \ on \ }S \\ 
f|_{\Omega }=g\text{ \ \ on \ }\Omega%
\end{array}%
\right.  \label{1.10.5}
\end{equation}%
Therewith,%
\begin{equation*}
\left( \left[ p_{0}+2\nu \partial _{x_{3}}(u_{0})_{3}+\lambda \text{div}%
(u_{0})\right] _{\Omega },{{\mathbf{U\cdot }\nabla }}w_{1}\right) _{\Omega
}=\left( \sigma (u_{0}),\epsilon (D_{\Omega }({{\mathbf{U\cdot }\nabla }}%
w_{1}))\right) _{\mathcal{O}}
\end{equation*}%
\begin{equation*}
+\left( \nabla p_{0},D_{\Omega }({{\mathbf{U\cdot }\nabla }}w_{1})\right) _{%
\mathcal{O}}+\left( p_{0},\text{div}(D_{\Omega }({{\mathbf{U\cdot }\nabla }}%
w_{1}))\right) _{\mathcal{O}}+\left( \text{div}\sigma (u_{0}),D_{\Omega }({{%
\mathbf{U\cdot }\nabla }}w_{1})\right) _{\mathcal{O}}
\end{equation*}%
\begin{equation*}
=\left( \sigma (u_{0}),\epsilon (D_{\Omega }({{\mathbf{U\cdot }\nabla }}%
w_{1}))\right) _{\mathcal{O}}+\eta \left( u_{0},D_{\Omega }({{\mathbf{U\cdot 
}\nabla }}w_{1})\right) _{\mathcal{O}}+\left( p_{0},\text{div}(D_{\Omega }({{%
\mathbf{U\cdot }\nabla }}w_{1}))\right) _{\mathcal{O}}
\end{equation*}%
\begin{equation}
+\lambda \left( u_{0},D_{\Omega }({{\mathbf{U\cdot }\nabla }}w_{1})\right) _{%
\mathcal{O}}-\left( {{\mathbf{U\cdot }\nabla }}u_{0},D_{\Omega }({{\mathbf{%
U\cdot }\nabla }}w_{1})\right) _{\mathcal{O}}+\left( \text{div}(\mathbf{U}%
)u_{0},D_{\Omega }({{\mathbf{U\cdot }\nabla }}w_{1})\right) _{\mathcal{O}}
\label{1.11}
\end{equation}%
Now, applying (\ref{1.11}) to RHS of (\ref{1.10}), and invoking (\ref{1.9})
we then have%
\begin{equation*}
\left\vert K_{1}\right\vert =\left\vert \left( {\left\{ \text{div}{{[U}%
_{1},U_{2}]{+\mathbf{U\cdot }\nabla }}\right\} }\left[ p_{0}+2\nu \partial
_{x_{3}}(u_{0})_{3}+\lambda \text{div}(u_{0})\right] _{\Omega },w_{1}\right)
_{\Omega }\right\vert
\end{equation*}%
\begin{equation}
\leq Cr(\left\Vert \mathbf{U}\right\Vert _{\ast })\left\{ \sigma
(u_{0}),\epsilon (u_{0}))_{\mathcal{O}}+\eta \left\Vert u_{0}\right\Vert _{%
\mathcal{O}}^{2}+\lambda \left\Vert \varphi \right\Vert _{\mathcal{H}%
}^{2}\right\}  \label{1.12}
\end{equation}%
where again $r(\cdot )$ and $\left\Vert \mathbf{U}\right\Vert _{\ast }$ are
given as in (\ref{poly}) and (\ref{normU}), respectively. Let us now
continue with $K_{2}:$%
\begin{equation*}
K_{2}=-\left( {\left\{ \text{div}{{[U}_{1},U_{2}]{+\mathbf{U\cdot }\nabla }}%
\right\} }\Delta ^{2}w_{1},w_{1}\right) _{\Omega }
\end{equation*}%
\begin{equation}
=\left( \Delta ^{2}w_{1},{{\mathbf{U\cdot }\nabla }}w_{1}\right) _{\Omega }
\label{1.13}
\end{equation}%
If we argue as in the estimates (\ref{use1})-(\ref{use2}) by replacing $%
h_{\alpha}$ with ${{\mathbf{U,}}}$ we then have%
\begin{equation}
\left( \Delta ^{2}w_{1},{{\mathbf{U\cdot }\nabla }}w_{1}\right) _{\Omega
}=~(\Delta w_{1},[\Delta ,{{\mathbf{U}}}\cdot \nabla ]w_{1})_{\Omega } 
\notag
\end{equation}%
\begin{equation}
-\frac{1}{2}\int_{\partial \Omega }({{\mathbf{U}}}\cdot \mathbf{\nu }%
)|\Delta w_{1}|^{2}d\partial \Omega -\frac{1}{2}\int_{\Omega }\text{div}({{%
\mathbf{U}}})|\Delta w_{1}|^{2}d\Omega  \label{1.14}
\end{equation}%
For the second term on RHS of (\ref{1.14}), let $\gamma (x)$ be a $C^{2}-$%
extension of the normal vector $\mathbf{\nu (x)}$ to the boundary of $\Omega
.$ Applying the multiplier $\gamma \cdot \nabla w_{1}$ to the structral
equation (\ref{1.2})$_{7},$ we get%
\begin{equation}
\left( \Delta ^{2}w_{1},\gamma {{\mathbf{\cdot }\nabla }}w_{1}\right)
_{\Omega }=\left( \left[ p_{0}+2\nu \partial _{x_{3}}(u_{0})_{3}+\lambda 
\text{div}(u_{0})\right] |_{\Omega },\gamma {{\mathbf{\cdot }\nabla }}%
w_{1}\right) _{\Omega }+\lambda (w_{2},\gamma {{\mathbf{\cdot }\nabla }}%
w_{1})_{\Omega }  \label{1.15}
\end{equation}%
Revoking the elliptic map (\ref{1.10.5}), we have%
\begin{equation*}
\left( \left[ p_{0}+2\nu \partial _{x_{3}}(u_{0})_{3}+\lambda \text{div}%
(u_{0})\right] |_{\Omega },\gamma {{\mathbf{\cdot }\nabla }}w_{1}\right)
_{\Omega }
\end{equation*}%
\begin{equation*}
=\left( \sigma (u_{0}),\epsilon (D_{\Omega }(\gamma {{\mathbf{\cdot }\nabla }%
}w_{1}))\right) _{\mathcal{O}}+\eta \left( u_{0},D_{\Omega }(\gamma {{%
\mathbf{\cdot }\nabla }}w_{1})\right) _{\mathcal{O}}+\left( p_{0},\text{div}%
(D_{\Omega }(\gamma {{\mathbf{\cdot }\nabla }}w_{1}))\right) _{\mathcal{O}}
\end{equation*}%
\begin{equation}
+\lambda \left( u_{0},D_{\Omega }(\gamma {{\mathbf{\cdot }\nabla }}%
w_{1})\right) _{\mathcal{O}}-\left( {{\mathbf{U\cdot }\nabla }}%
u_{0},D_{\Omega }(\gamma {{\mathbf{\cdot }\nabla }}w_{1})\right) _{\mathcal{O%
}}+\left( \text{div}(\mathbf{U})u_{0},D_{\Omega }({{\mathbf{U\cdot }\nabla }}%
w_{1})\right) _{\mathcal{O}}  \label{1.16}
\end{equation}%
Moreover, proceeding as in (\ref{1.14}), we get%
\begin{equation*}
\left( \Delta ^{2}w_{1},\gamma {{\mathbf{\cdot }\nabla }}w_{1}\right)
_{\Omega }=~(\Delta w_{1},[\Delta ,\gamma \cdot \nabla ]w_{1})_{\Omega } 
\end{equation*}%
\begin{equation}
-\frac{1}{2}\int_{\partial \Omega }|\Delta w_{1}|^{2}d\partial \Omega -\frac{%
1}{2}\int_{\Omega }\text{div}(\gamma )|\Delta w_{1}|^{2}d\Omega  \label{1.17}
\end{equation}%
Now, applying (\ref{1.16}), (\ref{1.17}) to (\ref{1.15}), using (\ref%
{commest}) (replacing $h_{\alpha}$ with $\gamma$) and subsequently
re-invoking (\ref{1.9}), we obtain%
\begin{equation}
\int_{\partial \Omega }|\Delta w_{1}|^{2}d\partial \Omega \leq Cr(\left\Vert 
\mathbf{U}\right\Vert _{\ast })\left\{ \sigma (u_{0}),\epsilon (u_{0}))_{%
\mathcal{O}}+\eta \left\Vert u_{0}\right\Vert _{\mathcal{O}}^{2}+\lambda
\left\Vert \varphi \right\Vert _{\mathcal{H}}^{2}\right\}  \label{1.18}
\end{equation}%
Combining now (\ref{1.13}), (\ref{1.14}), (\ref{1.18}) and (\ref{1.9}), we
have 
\begin{equation*}
\left\vert K_{2}\right\vert =\left\vert \left( {\left\{ \text{div}{{[U}%
_{1},U_{2}]{+\mathbf{U\cdot }\nabla }}\right\} }\Delta
^{2}w_{1},w_{1}\right) _{\Omega }\right\vert
\end{equation*}%
\begin{equation}
\leq Cr(\left\Vert \mathbf{U}\right\Vert _{\ast })\left\{ \sigma
(u_{0}),\epsilon (u_{0}))_{\mathcal{O}}+\eta \left\Vert u_{0}\right\Vert _{%
\mathcal{O}}^{2}+\lambda \left\Vert \varphi \right\Vert _{\mathcal{H}%
}^{2}\right\}  \label{1.19}
\end{equation}%
Hence, the second term of (\ref{1.5}) can be handled by%
\begin{equation*}
\left\vert \left( {\left\{ \text{div}{{[U}_{1},U_{2}]{+\mathbf{U\cdot }%
\nabla }}\right\} }\left[ p_{0}+2\nu \partial _{x_{3}}(u_{0})_{3}+\lambda 
\text{div}(u_{0})-\Delta ^{2}w_{1}\right] _{\Omega },w_{1}\right) _{\Omega
}\right\vert
\end{equation*}%
\begin{equation*}
\leq \left\vert K_{1}\right\vert +\left\vert K_{2}\right\vert
\end{equation*}%
\begin{equation}
\leq Cr(\left\Vert \mathbf{U}\right\Vert _{\ast })\left\{ \sigma
(u_{0}),\epsilon (u_{0}))_{\mathcal{O}}+\eta \left\Vert u_{0}\right\Vert _{%
\mathcal{O}}^{2}+\lambda \left\Vert \varphi \right\Vert _{\mathcal{H}%
}^{2}\right\}  \label{1.19a}
\end{equation}%
Also, for the third and fourth terms of (\ref{1.5}): 
\begin{equation*}
\left( \Delta \lbrack \mathbf{U}\mathbb{\cdot }\nabla w_{1}],\Delta
w_{1}\right) _{\Omega }+\left( {\nabla }^{\ast }(\mathbb{\nabla \cdot }(%
\mathbf{U}\mathbb{\cdot }\nabla w_{1})\mathbb{)},\Delta w_{1}\right)
_{\Omega }
\end{equation*}%
\begin{equation*}
=\left( \mathbf{U}\mathbb{\cdot }\nabla (\Delta w_{1}),\Delta w_{1}\right)
_{\Omega }+([\Delta ,\mathbf{U}\cdot \nabla ]w_{1},\Delta w_{1})_{\Omega
}+\left( \nabla \lbrack \mathbf{U}\mathbb{\cdot }\nabla w_{1}],\nabla
(\Delta w_{1})\right) _{\Omega }
\end{equation*}%
\begin{equation*}
=\frac{1}{2}\int_{\partial \Omega }({{\mathbf{U}}}\cdot \mathbf{\nu }%
)|\Delta w_{1}|^{2}d\partial \Omega -\frac{1}{2}\int_{\Omega }\text{div}({{%
\mathbf{U}}})|\Delta w_{1}|^{2}d\Omega
\end{equation*}%
\begin{equation*}
+([\Delta ,\mathbf{U}\cdot \nabla ]w_{1},\Delta w_{1})_{\Omega }-\left( 
\mathbf{U}\mathbb{\cdot }\nabla w_{1},\Delta ^{2}w_{1}\right) _{\Omega }
\end{equation*}%
Proceeding as done above, we then have%
\begin{equation*}
\left\vert \left( \Delta \lbrack \mathbf{U}\mathbb{\cdot }\nabla
w_{1}],\Delta w_{1}\right) _{\Omega }+\left( {\nabla }^{\ast }(\mathbb{%
\nabla \cdot }(\mathbf{U}\mathbb{\cdot }\nabla w_{1})\mathbb{)},\Delta
w_{1}\right) _{\Omega }\right\vert
\end{equation*}%
\begin{equation}
\leq Cr(\left\Vert \mathbf{U}\right\Vert _{\ast })\left\{ \sigma
(u_{0}),\epsilon (u_{0}))_{\mathcal{O}}+\eta \left\Vert u_{0}\right\Vert _{%
\mathcal{O}}^{2}+\lambda \left\Vert \varphi \right\Vert _{\mathcal{H}%
}^{2}\right\}  \label{1.19.5}
\end{equation}%
Finally, if we apply the estimates (\ref{1.9}), (\ref{1.19a}) and (\ref%
{1.19.5}) to RHS of (\ref{1.5}), we arrive at%
\begin{equation*}
\lambda \left\Vert \varphi \right\Vert _{\mathcal{H}}^{2}+\sigma
(u_{0}),\epsilon (u_{0}))_{\mathcal{O}}+\eta \left\Vert u_{0}\right\Vert _{%
\mathcal{O}}^{2}
\end{equation*}%
\begin{equation*}
\leq C\left\Vert \mathbf{U}\right\Vert _{*}\left\{ \lambda \left\Vert
\varphi \right\Vert _{\mathcal{H}}^{2}+(\sigma (u_{0}),\epsilon (u_{0}))_{%
\mathcal{O}}+\eta \left\Vert u_{0}\right\Vert _{\mathcal{O}}^{2}\right\}
\end{equation*}%
For $\left\Vert \mathbf{U}\right\Vert _{\ast }$ small enough-independent of $%
\lambda >0$- we infer that the solution $\varphi $ of (\ref{1.1}) is zero
which concludes the proof of Lemma \ref{denserange}.
\end{proof}
\newline\newline
\underline{\textbf{STEP (M-II):}} We continue with showing that $[\lambda I-(%
\mathcal{A}+B)]$ is a closed operator. For this, it will be enough to prove
the following lemma:

\begin{lemma}
\label{closed} The operator $\mathcal{A}+B:D(\mathcal{A}+B)\cap H_{N}^{\bot
}\rightarrow H_{N}^{\bot }$ is closed.
\end{lemma}

\begin{proof}
Let $\left\{ \varphi _{n}\right\} =\left\{ \left[ p_{0n},u_{0n},w_{1n},w_{2n}%
\right] \right\} \subseteq D(\mathcal{A}+B)\cap H_{N}^{\bot }$ satisfy 
\begin{eqnarray*}
\mathbf{\ }\varphi _{n} &\rightarrow &\varphi \ \ \ \text{in}\ \ \
H_{N}^{\bot },\  \\
\ (\mathcal{A}+B)\varphi _{n} &\rightarrow &\varphi ^{\ast }\ \ \ \text{in \ 
}\ H_{N}^{\bot }
\end{eqnarray*}%
We must show that $\varphi \in D(\mathcal{A}+B)\cap H_{N}^{\bot },$ and $(%
\mathcal{A}+B)\varphi =\varphi ^{\ast }.$ To start, via the relation (\ref%
{dissest}) in Lemma \ref{diss}, we have%
\begin{equation*}
\frac{(\sigma (u_{0m}-u_{0n}),\epsilon (u_{0m}-u_{0n}))_{\mathcal{O}}}{4}%
\leq -\text{Re}(([\mathcal{A}+B](\varphi _{m}-\varphi _{n},\varphi
_{m}-\varphi _{n}))_{H_{N}^{\bot }}
\end{equation*}%
from which we infer that 
\begin{equation}
u_{0n}\rightarrow u\text{ \ \ in \ \ }H^{1}(\mathcal{O})  \label{2.2}
\end{equation}%
Assume that for $\varphi _{n}^{\ast }=\left\{ \left[ p_{0n}^{\ast
},u_{0n}^{\ast },w_{1n}^{\ast },w_{2n}^{\ast }\right] \right\} \subseteq
H_{N}^{\bot }$ 
\begin{equation}
(\mathcal{A}+B)\varphi _{n}=\varphi _{n}^{\ast }  \label{2.1}
\end{equation}%
In PDE terms this gives%
\begin{equation}
\left\{ 
\begin{array}{c}
-\mathbf{U}\nabla p_{0n}-\text{div}(u_{0n})-\text{div}(\mathbf{U}%
)p_{0n}=p_{0n}^{\ast }\text{ \ \ in \ \ }\mathcal{O} \\ 
-\nabla p_{0n}+\text{div}\sigma (u_{0n})-\eta u_{0n}-\mathbf{U}\nabla
u_{0n}=u_{0n}^{\ast }\text{ \ \ in \ \ }\mathcal{O} \\ 
w_{2n}+{{\mathbf{U\cdot }\nabla }}w_{1n}=w_{1n}^{\ast }\text{ \ \ in \ \ }%
\Omega \\ 
p_{0n}-[2\nu \partial _{x_{3}}(u_{0n})_{3}+\lambda \text{div}%
(u_{0n})\}|_{\Omega }-\Delta ^{2}w_{1n}=w_{2n}^{\ast }\text{ \ \ in \ \ }%
\Omega%
\end{array}%
\right.  \label{2.1.1}
\end{equation}%
If we read off the first equation in (\ref{2.1.1}) to have 
\begin{equation*}
\mathbf{U}\nabla p_{0n}=-\text{div}(u_{0n})-\text{div}(\mathbf{U}%
)p_{0n}-p_{0n}^{\ast }
\end{equation*}%
and take upon the limit when $n\rightarrow \infty $ we get%
\begin{equation}
\mathbf{U}\nabla p_{0}=[-\text{div}(u_{0})-\text{div}(\mathbf{U}%
)p_{0}-p_{0}^{\ast }]\in L^{2}(\mathcal{O})  \label{2.3}
\end{equation}%
Moreover, using the third equation in (\ref{2.1.1}), we have%
\begin{equation}
w_{2}=\lim_{n\rightarrow \infty }w_{2n}=\lim_{n\rightarrow \infty
}[w_{1n}^{\ast }-{{\mathbf{U\cdot }\nabla }}w_{1n}]=[w_{1}^{\ast }-{{\mathbf{%
U\cdot }\nabla }}w_{1}]\in H_{0}^{1}(\Omega )  \label{2.4}
\end{equation}%
In addition, from the domain criteria for $(\mathcal{A}+B)$, we have $%
u_{0n}=\mu _{0n}+\widetilde{\mu }_{0n}$, where $\mu _{0n}\in \mathbf{V}_{0}$
and $\widetilde{\mu }_{0n}\in H^{1}(\mathcal{O})$ satisfies%
\begin{equation*}
\widetilde{\mu }_{0n}=%
\begin{cases}
0 & ~\text{ on }~S \\ 
(w_{2n}+\mathbf{U}\cdot \nabla w_{1n})\mathbf{n} & ~\text{ on}~\Omega%
\end{cases}%
\end{equation*}%
Since $\mathbf{V}_{0}$ is closed, then by (\ref{2.2}), (\ref{2.4}) and the
Sobolev Trace Theorem, we have \begin{equation} u_{0}=\mu _{0}+\widetilde{%
\mu }_{0}, \label{2.5} \end{equation}where $\mu _{0}\in \mathbf{V}_{0}$ and $%
\widetilde{\mu }_{0}\in H^{1}(\mathcal{O})$ satisfies\begin{equation*} 
\widetilde{\mu }_{0}=\begin{cases} 0 & ~\text{ on }~S \\ (w_{2}+\mathbf{U}%
\cdot \nabla w_{1})\mathbf{n} & ~\text{ on}~\Omega \end{cases}\end{equation*}%
Furthermore, we recall the form of the adjoint $(\mathcal{A}+B)^{\ast }:$ $D(%
\mathcal{A}+B)^{\ast }\cap H_{N}^{\bot }\subset H_{N}^{\bot }\rightarrow
H_{N}^{\bot }$ in (\ref{adj-AplusB}) and given arbitrary $\Phi \in \mathcal{D%
}(\mathcal{O})$ we will have then $[0,\Phi ,0,0]\in D(\mathcal{A}+B)^{\ast
}\cap H_{N}^{\bot }.$ Therewith, we have \begin{equation*} \left( \varphi ,(%
\mathcal{A}+B)^{\ast }\left[ \begin{array}{c} 0 \\ \Phi \\ 0 \\ 0\end{array}%
\right] \right) _{\mathcal{H}}=\lim_{n\rightarrow \infty }\left( \varphi
_{n},(\mathcal{A}+B)^{\ast }\left[ \begin{array}{c} 0 \\ \Phi \\ 0 \\ 0\end{%
array}\right] \right) _{\mathcal{H}} \end{equation*}\begin{equation*}
=\lim_{n\rightarrow \infty }\left( (\mathcal{A}+B)\varphi _{n},\left[ \begin{%
array}{c} 0 \\ \Phi \\ 0 \\ 0\end{array}\right] \right) _{\mathcal{H}%
}=\left( (\varphi ^{\ast },\left[ \begin{array}{c} 0 \\ \Phi \\ 0 \\ 0\end{%
array}\right] \right) _{\mathcal{H}}, \end{equation*}or\begin{equation*}
(p_{0},\text{div}(\Phi ))_{\mathcal{O}}+(u_{0},\text{div}\sigma (\Phi )-\eta
\Phi +\mathbf{U}\cdot \nabla \Phi +\text{div}(\mathbf{U})\Phi )_{\mathcal{O}%
}=(u_{0}^{\ast },\Phi )_{\mathcal{O}} \end{equation*}Upon an integration by
parts this relation now becomes\begin{equation*} -(\nabla p_{0},\Phi )_{%
\mathcal{O}}+(\text{div}\sigma (u_{0}),\Phi )_{\mathcal{O}}-\eta (u_{0},\Phi
)_{\mathcal{O}}-(\mathbf{U}\cdot \nabla u_{0},\Phi )_{\mathcal{O}%
}=(u_{0}^{\ast },\Phi )_{\mathcal{O}},\text{ \ \ }\forall \text{ }\Phi \in 
\mathcal{D}(\mathcal{O}) \end{equation*}Applying a density argument to the
above relation gives\begin{equation} -\nabla p_{0}+\text{div}\sigma
(u_{0})-\eta u_{0}-\mathbf{U}\cdot \nabla u_{0}=u_{0}^{\ast }\in L^{2}(%
\mathcal{O}) \label{2.6} \end{equation}A further integration by parts
assigns a meaning to the trace $[\sigma (u_{0})\mathbf{n}-p_{0}\mathbf{n}%
]_{\partial \mathcal{O}}$ in $H^{-\frac{1}{2}}-$sense. What is more: If $%
\gamma _{0}^{+}(\cdot )\in L(H^{\frac{1}{2}}(\partial \mathcal{O}),H^{1}(%
\mathcal{O}))$ is the right inverse of Sobolev Trace Map $\gamma _{0}(\cdot
)=(\cdot )|_{\partial \mathcal{O}},$ then for every $g\in H^{\frac{1}{2}%
}(\partial \mathcal{O}),$ we have\begin{equation*} \left\langle \lbrack
\sigma (u_{0})\mathbf{n}-p_{0}\mathbf{n}]_{\partial \mathcal{O}%
},g\right\rangle _{\partial \mathcal{O}}=(\sigma (u_{0}),\epsilon (\gamma
_{0}^{+}(g)))_{\mathcal{O}}+(\text{div}\sigma (u_{0}),\gamma _{0}^{+}(g))_{%
\mathcal{O}} \end{equation*}\begin{equation*} -(p_{0},\text{div}\gamma
_{0}^{+}(g))_{\mathcal{O}}-(\nabla p_{0},\gamma _{0}^{+}(g))_{\mathcal{O}} 
\end{equation*}\begin{equation*} =(\sigma (u_{0}),\epsilon (\gamma
_{0}^{+}(g)))_{\mathcal{O}}+\eta (u_{0},\gamma _{0}^{+}(g))_{\mathcal{O}}+(%
\mathbf{U}\cdot \nabla u_{0},\gamma _{0}^{+}(g))_{\mathcal{O}} \end{equation*%
}\begin{equation*} +(u_{0}^{\ast },\gamma _{0}^{+}(g))_{\mathcal{O}}-(p_{0},%
\text{div}\gamma _{0}^{+}(g))_{\mathcal{O}} \end{equation*}\begin{equation*}
=\lim_{n\rightarrow \infty }[(\sigma (u_{0n}),\epsilon (\gamma
_{0}^{+}(g)))_{\mathcal{O}}+\eta (u_{0n},\gamma _{0}^{+}(g))_{\mathcal{O}}+(%
\mathbf{U}\cdot \nabla u_{0n},\gamma _{0}^{+}(g))_{\mathcal{O}} \end{%
equation*}\begin{equation*} +(u_{0n}^{\ast },\gamma _{0}^{+}(g))_{\mathcal{O}%
}-(p_{0n},\text{div}\gamma _{0}^{+}(g))_{\mathcal{O}}] \end{equation*}\begin{%
equation*} =\lim_{n\rightarrow \infty }\left\langle [\sigma (u_{0n})\mathbf{n%
}-p_{0n}\mathbf{n}]_{\partial \mathcal{O}},g\right\rangle _{\partial 
\mathcal{O}} \end{equation*}That is \begin{equation} \lbrack \sigma (u_{0n})%
\mathbf{n}-p_{0n}\mathbf{n}]_{\partial \mathcal{O}}\rightarrow \lbrack
\sigma (u_{0})\mathbf{n}-p_{0}\mathbf{n}]_{\partial \mathcal{O}}\text{ \ \
in \ \ }H^{\frac{1}{2}}(\partial \mathcal{O}) \label{2.7} \end{equation}The
last relation in turn allows us to pass to limit in (\ref{2.1.1})$_{4}$, and
we get\begin{equation} \lbrack p_{0}-(2\nu \partial
_{x_{3}}(u_{0})_{3}+\lambda \text{div}(u_{0}))]|_{\Omega }-\Delta
^{2}w_{1}=w_{2}^{\ast }\in L^{2}(\Omega ) \label{2.7.5} \end{equation}%
Lastly, from (\ref{2.5}) and (\ref{2.6}) and the Lax-Milgram Theorem, the
flow component $u_{0}=\mu _{0}+\widetilde{\mu }_{0}$ can be characterized
via the solution $\mu _{0}\in \mathbf{V}_{0}$ of the following variational
problem for all $\chi \in \mathbf{V}_{0}$: \begin{equation*} (\sigma (\mu
_{0}),\epsilon (\chi ))_{\mathcal{O}}+\eta (\mu _{0},\chi )_{\mathcal{O}%
}=-(\sigma (\widetilde{\mu }_{0}),\epsilon(\chi) )_{\mathcal{O}}-\eta (%
\widetilde{\mu }_{0},\chi )_{\mathcal{O}} \end{equation*}\begin{equation*}
+(p_{0},\text{div}(\chi ))_{\mathcal{O}}-(\mathbf{U}\cdot \nabla u_{0},\chi
)_{\mathcal{O}}-(u_{0}^{\ast },\chi )_{\mathcal{O}} \end{equation*}An
integration by parts with respect to this relation now gives for all $\chi
\in V_{0},$\begin{equation*} -(\text{div}\sigma (u_{0}),\chi )_{\mathcal{O}%
}+\eta (u_{0},\chi )_{\mathcal{O}}+\left\langle \sigma (u_{0})\mathbf{n,}%
\chi \right\rangle _{\partial \mathcal{O}} \end{equation*}\begin{equation*}
=-(\nabla p_{0},\chi )_{\mathcal{O}}+\left\langle p_{0}\mathbf{n,}\chi
\right\rangle _{\partial \mathcal{O}}-(\mathbf{U}\cdot \nabla u_{0},\chi )_{%
\mathcal{O}}-(u_{0}^{\ast },\chi )_{\mathcal{O}} \end{equation*}or after
using (\ref{2.6})\begin{equation*} \left\langle \sigma (u_{0})\mathbf{n}%
-p_{0}\mathbf{n,}\chi \right\rangle _{\partial \mathcal{O}}=0,\text{ \ for
every }\chi \in V_{0} \end{equation*}which gives in the sense of
distributions\begin{equation} \lbrack \sigma (u_{0})\mathbf{n}-p_{0}\mathbf{n%
}]\cdot \tau =0,\text{ \ }\forall \text{ }\tau \in TH^{\frac{1}{2}}(\partial 
\mathcal{O}) \label{2.8} \end{equation}Hence, the estimates (\ref{2.2})-(%
\ref{2.8}) now give the desired conclusion and completes the proof of Lemma 
\ref{closed}. \end{proof} \newline
\newline
\underline{\textbf{STEP (M-III):}} Lastly, we prove the following fact: 

\begin{lemma} \label{inv} For given $\lambda >0,$ we have the existence of a
constant $\varrho >0$ such that for all $\varphi \in D(\mathcal{A}+B)\cap
H_{N}^{\bot } $\begin{equation} \left\Vert \left\vert \lbrack \lambda I-(%
\mathcal{A}+B)]\varphi \right\vert \right\Vert _{H_{N}^{\bot }} \geq \varrho
\left\Vert \left\vert \varphi \right\vert \right\Vert _{H_{N}^{\bot }} 
\label{36} \end{equation}where the norm $\left\Vert \left\vert \cdot
\right\vert \right\Vert _{H_{N}^{\bot }}$ is defined in (\ref{specnorm}). 
\end{lemma} \begin{proof} Using the estimate (\ref{dissest}) in Lemma \ref{%
diss}, we have for given $\lambda >0,$\begin{equation*} \left( \left(
\lbrack \lambda I-(\mathcal{A}+B)]\varphi ,\varphi \right) \right)
_{H_{N}^{\bot }} \end{equation*}\begin{equation*} \geq \lambda \left\Vert
\left\vert \varphi \right\vert \right\Vert _{H_{N}^{\bot
}}^{2}+C_{1}\left\Vert u_{0}\right\Vert _{H^{1}(\mathcal{O})}^{2}+\frac{%
\epsilon }{2}\left[ \left\Vert p_{0}\right\Vert _{\mathcal{O}%
}^{2}+\left\Vert \Delta w_{1}\right\Vert _{\Omega }^{2}\right] \end{equation*%
}\begin{equation} \geq \lambda \left\Vert \left\vert \varphi \right\vert
\right\Vert _{H_{N}^{\bot }}^{2}+(C_{1}-\frac{\epsilon }{2})\left\Vert
u_{0}\right\Vert _{H^{1}(\mathcal{O})}^{2}+\frac{\epsilon }{2}\left[
\left\Vert p_{0}\right\Vert _{\mathcal{O}}^{2}+\left\Vert u_{0}\right\Vert _{%
\mathcal{O}}^{2}+\left\Vert \Delta w_{1}\right\Vert _{\Omega }^{2}\right] 
\label{38} \end{equation}With respect to the RHS: we firstly add and
subtract, so as to have\begin{equation*} \left\Vert u_{0}\right\Vert _{%
\mathcal{O}}^{2}=\left\Vert \left[ u_{0}-\alpha D(g\cdot \nabla
w_{1})e_{3}+\xi \nabla \psi (p_{0},w_{1})\right] +\alpha D(g\cdot \nabla
w_{1})e_{3}-\xi \nabla \psi (p_{0},w_{1})\right\Vert _{\mathcal{O}}^{2} \end{%
equation*}\begin{equation*} =\left\Vert \left[ u_{0}-\alpha D(g\cdot \nabla
w_{1})e_{3}+\xi \nabla \psi (p_{0},w_{1})\right] \right\Vert _{\mathcal{O}%
}^{2} \end{equation*}\begin{equation*} +2\text{Re}\left( u_{0}-\alpha
D(g\cdot \nabla w_{1})e_{3}+\xi \nabla \psi (p_{0},w_{1}),\alpha D(g\cdot
\nabla w_{1})e_{3}-\xi \nabla \psi (p_{0},w_{1})\right) _{\mathcal{O}} \end{%
equation*}\begin{equation} +\left\Vert \alpha D(g\cdot \nabla
w_{1})e_{3}-\xi \nabla \psi (p_{0},w_{1})\right\Vert _{\mathcal{O}}^{2} 
\label{39} \end{equation}By using Holder-Young Inequalities we get\begin{%
equation*} \left\Vert u_{0}\right\Vert _{\mathcal{O}}^{2}\geq (1-\delta
)\left\Vert u_{0}-\alpha D(g\cdot \nabla w_{1})e_{3}+\xi \nabla \psi
(p_{0},w_{1})\right\Vert _{\mathcal{O}}^{2} \end{equation*}\begin{equation}
+(1-C_{\delta })\left\Vert \alpha D(g\cdot \nabla w_{1})e_{3}-\xi \nabla
\psi (p_{0},w_{1})\right\Vert _{\mathcal{O}}^{2} \label{40} \end{equation}%
Using the boundedness of the maps $D(\cdot )$ and $\psi (\cdot ,\cdot )$
defined in (\ref{liftnorm}) and (\ref{psireg}), respectively we then have 
\begin{equation*} \left\Vert u_{0}\right\Vert _{\mathcal{O}}^{2}\geq
(1-\delta )\left\Vert u_{0}-\alpha D(g\cdot \nabla w_{1})e_{3}+\xi \nabla
\psi (p_{0},w_{1})\right\Vert _{\mathcal{O}}^{2} \end{equation*}\begin{%
equation} +C_{2}(1-C_{\delta })\left[ \left\Vert \mathbf{U}\right\Vert
_{\ast }^{2}+\xi ^{2}\right] \left\Vert \Delta w_{1}\right\Vert _{\Omega
}^{2} \label{41} \end{equation}Now, applying (\ref{41}) to the RHS of (\ref{%
38}), we get\begin{equation*} \left( \left( \lbrack \lambda I-(\mathcal{A}%
+B)]\varphi ,\varphi \right) \right) _{H_{N}^{\bot }}\geq \lambda \left\Vert
\left\vert \varphi \right\vert \right\Vert _{H_{N}^{\bot }}^{2}+(C_{1}-\frac{%
\epsilon }{2})\left\Vert u_{0}\right\Vert _{H^{1}(\mathcal{O})}^{2} \end{%
equation*}\begin{equation*} +\frac{\epsilon }{2}\{\left\Vert
p_{0}\right\Vert _{\mathcal{O}}^{2}+(1-\delta )\left\Vert u_{0}-\alpha
D(g\cdot \nabla w_{1})e_{3}+\xi \nabla \psi (p_{0},w_{1})\right\Vert _{%
\mathcal{O}}^{2} \end{equation*}\begin{equation} +\left[ 1+C_{2}(1-C_{\delta
})\left[ \left\Vert \mathbf{U}\right\Vert _{\ast }^{2}+\xi ^{2}\right] %
\right] \left\Vert \Delta w_{1}\right\Vert _{\Omega }^{2}\} \label{42} \end{%
equation}If we take now$\left\Vert \mathbf{U}\right\Vert _{\ast }$ so small
such that \begin{equation*} \left\Vert \mathbf{U}\right\Vert _{\ast
}^{2}+\xi ^{2}<\frac{1}{2C_{2}(C_{\delta }-1)}, \end{equation*}we then have%
\begin{equation*} \left( \left( \lbrack \lambda I-(\mathcal{A}+B)]\varphi
,\varphi \right) \right) _{H_{N}^{\bot }}\geq \lambda \left\Vert \left\vert
\varphi \right\vert \right\Vert _{H_{N}^{\bot }}^{2}+(C_{1}-\frac{\epsilon }{%
2})\left\Vert u_{0}\right\Vert _{H^{1}(\mathcal{O})}^{2} \end{equation*}%
\begin{equation*} +\frac{\epsilon }{2}\left\{ \left\Vert p_{0}\right\Vert _{%
\mathcal{O}}^{2}+(1-\delta )\left\Vert u_{0}-\alpha D(g\cdot \nabla
w_{1})e_{3}+\xi \nabla \psi (p_{0},w_{1})\right\Vert _{\mathcal{O}}^{2}+%
\frac{1}{2}\left\Vert \Delta w_{1}\right\Vert _{\Omega }^{2}\right\} \end{%
equation*}\begin{equation*} \geq \frac{\epsilon }{2}\left\{ \left\Vert
p_{0}\right\Vert _{\mathcal{O}}^{2}+(1-\delta )\left\Vert u_{0}-\alpha
D(g\cdot \nabla w_{1})e_{3}+\xi \nabla \psi (p_{0},w_{1})\right\Vert _{%
\mathcal{O}}^{2}+\frac{1}{2}\left\Vert \Delta w_{1}\right\Vert _{\Omega
}^{2}\right\} \end{equation*}\begin{equation} +\lambda \left\Vert
w_{2}+h_{\alpha }\cdot \nabla w_{1}+\xi w_{1}\right\Vert _{\mathcal{O}}^{2} 
\label{45} \end{equation}Using Cauchy-Schwarz now we obtain\begin{equation*}
\left\Vert \left\vert \lbrack \lambda I-(\mathcal{A}+B)]\varphi \right\vert
\right\Vert _{H_{N}^{\bot }}\left\Vert \left\vert \varphi \right\vert
\right\Vert _{H_{N}^{\bot }} \end{equation*}\begin{equation*} \geq \frac{%
\epsilon }{2}\left\{ \left\Vert p_{0}\right\Vert _{\mathcal{O}%
}^{2}+(1-\delta )\left\Vert u_{0}-\alpha D(g\cdot \nabla w_{1})e_{3}+\xi
\nabla \psi (p_{0},w_{1})\right\Vert _{\mathcal{O}}^{2}+\frac{1}{2}%
\left\Vert \Delta w_{1}\right\Vert _{\Omega }^{2}\right\} \end{equation*}%
\begin{equation} +\lambda \left\Vert w_{2}+h_{\alpha }\cdot \nabla w_{1}+\xi
w_{1}\right\Vert _{\mathcal{O}}^{2} \label{46} \end{equation}which gives the
desired estimate (\ref{36}), with therein \begin{equation*} \varrho =\min
\left\{ \frac{\epsilon }{4},\lambda \right\} \end{equation*}and finishes the
proof of Lemma \ref{inv}. \end{proof} Now, combining Lemma \ref{denserange},
Lemma \ref{closed}\ and Lemma \ref{inv} gives that the map $[\lambda I-(%
\mathcal{A}+B)]$ satisfies the requirements of Lemma \ref{pazy} in Appendix
which, in turn, yields that \begin{equation*} \lbrack \lambda I-(\mathcal{A}%
+B)]^{-1}\in \mathcal{L}(H_{N}^{\bot }) \end{equation*}and the range
condition (\ref{range}) holds. This finishes the proof of Lemma \ref{md}. 

By Lemma \ref{diss} and Lemma \ref{md}, we have the desired contraction semigroup generation with respect to the special inner product $((\cdot ,\cdot))_{H_{N}^{\bot }}.$ Hence we have the asserted wellposedness statement of Theorem \ref{wp}. 

Moreover, form the values of the parameters $\alpha$ and $\xi$ in (\ref{alpha}) and (\ref{33}), respectively, as well as the definition of $((\cdot ,\cdot))_{H_{N}^{\bot }}$ in (\ref{innpro}), we infer that ${e^{(\mathcal{A}+B)t}}$ is uniformly bounded in time, in the standard $\mathcal{H}-$norm. In fact, given $\phi^{\ast}=[p^*,u^*,w_1^*,w_2^*]\in H_{N}^{\bot },$ set
\begin{equation}
\phi(t)=\left[ 
\begin{array}{c}
p(t) \\ 
u(t) \\ 
w_{1}(t) \\ 
w_{2}(t)%
\end{array}%
\right] =e^{(\mathcal{A}+B)t}\left[ 
\begin{array}{c}
p^{*} \\ 
u^{*} \\ 
w_1^* \\ 
w_2^*%
\end{array}%
\right] 
\end{equation}
Then,
\begin{equation*}
\left\Vert \phi(t) \right\Vert _{\mathcal{H}}^2=\left\Vert p\right\Vert _{\mathcal{O}}^{2}+\left\Vert u\right\Vert _{\mathcal{O}}^{2}+\left\Vert \Delta w_1\right\Vert _{\Omega}^{2}+\left\Vert w_2\right\Vert _{\Omega}^{2}
\end{equation*}
\begin{equation*}
\leq C\Big[\left\Vert p\right\Vert _{\mathcal{O}}^{2}+ \left\Vert u-\alpha
D(g\cdot \nabla w_{1})e_{3}+\xi \nabla \psi (p,w_{1})\right\Vert _{%
\mathcal{O}}^{2}+\alpha^2 \left\Vert 
D(g\cdot \nabla w_{1})e_{3}\right\Vert _{%
\mathcal{O}}^{2}
\end{equation*}
\begin{equation*}
+\xi^2 \left\Vert \nabla \psi (p,w_{1})\right\Vert _{%
\mathcal{O}}^{2}+\left\Vert \Delta w_1\right\Vert _{\Omega}^{2}+\left\Vert w_{2}+h_{\alpha }\cdot \nabla w_{1}+\xi w_{1}\right\Vert
_{\Omega }^{2} +\left\Vert h_{\alpha }\cdot \nabla w_{1}+\xi w_{1}\right\Vert
_{\Omega }^{2} \Big]
\end{equation*}
\begin{equation*}
\leq C\Big[\left\Vert \left\vert  e^{(\mathcal{A}+B)t}\phi^{\ast} \right\vert \right\Vert _{H_{N}^{\bot }}^2+\alpha^2 \left\Vert D(g\cdot \nabla w_{1})e_{3}\right\Vert _{%
\mathcal{O}}^{2}+\xi^2 \left\Vert \nabla \psi (p,w_{1})\right\Vert _{%
\mathcal{O}}^{2} +\left\Vert h_{\alpha }\cdot \nabla w_{1}+\xi w_{1}\right\Vert
_{\Omega }^{2} \Big].
\end{equation*}
Using the fact that $e^{(\mathcal{A}+B)t}$ is a contraction semigroup on $H_{N}^{\bot }$ with respect to the norm $\left\Vert \left\vert  \cdot \right\vert \right\Vert _{H_{N}^{\bot }},$ then combining this fact with (\ref{specnorm}), we have
 \begin{equation*}
\left\Vert \phi(t) \right\Vert _{\mathcal{H}}^2\leq  C[\left\Vert \mathbf{U}\right\Vert _{\ast
}^{2}+\xi ^{2}]\left\Vert \phi(t) \right\Vert _{\mathcal{H}}^2+C_1 \left\Vert \phi^* \right\Vert _{\mathcal{H}}^2
\end{equation*}
For $\left\Vert \mathbf{U}\right\Vert _{\ast}$ small enough, we then have 
 \begin{equation*}
\left\Vert \phi(t) \right\Vert _{\mathcal{H}}\leq  C^* \left\Vert \phi^* \right\Vert _{\mathcal{H}},\text{ \ \ \ for all \ } t>0.
\end{equation*}
This concludes the proof of Theorem \ref{wp}.
\section{Appendix}

In this section we will provide some useful lemmas that are critically used
in this manuscript. In reference to problem (\ref{1})-(\ref{3}), we start with defining the adjoint operator $(\mathcal{A}+B)^{\ast }:$ $D((\mathcal{A}%
+B)^{\ast })\cap H_{N}^{\bot }\subset H_{N}^{\bot }\rightarrow H_{N}^{\bot }$ of the semigroup generator $\mathcal{A}+B$ in the following lemma:

\begin{lemma}
\label{adj} The adjoint operator of the generator $(\mathcal{A}+B)$ (given via
(\ref{AAA})-(\ref{feedbackB})) is defined as%
\begin{equation*}
(\mathcal{A}+B)^{\ast }=\mathcal{A}^{\ast }+B^{\ast }
\end{equation*}%
\begin{equation*}
=\left[ 
\begin{array}{cccc}
\mathbf{U}\mathbb{\cdot }\nabla (\cdot ) & \text{div}(\cdot ) & 0 & 0 \\ 
\mathbb{\nabla (\cdot )} & \text{div}\sigma (\cdot )-\eta I+\mathbf{U}%
\mathbb{\cdot \nabla (\cdot )} & 0 & 0 \\ 
0 & 0 & 0 & -I \\ 
-\left[ \cdot \right] _{\Omega } & -\left[ 2\nu \partial _{x_{3}}(\cdot
)_{3}+\lambda \text{div}(\cdot )\right] _{\Omega } & \Delta ^{2} & 0%
\end{array}%
\right] 
\end{equation*}%
\begin{equation*}
+\left[ 
\begin{array}{cccc}
\text{div}(\mathbf{U)}(\cdot ) & 0 & 0 & 0 \\ 
0 & \text{div}(\mathbf{U)}(\cdot ) & 0 & 0 \\ 
{{\mathring{A}}^{-1}}\left\{ \text{div}{{([U}_{1},U_{2}]{)+\mathbf{U\cdot }%
\nabla )}}\right\} {(\cdot )|}_{\Omega } & {{\mathring{A}}^{-1}\left\{ \text{%
div}{{[U}_{1},U_{2}]{+\mathbf{U\cdot }\nabla }}\right\} }\left[ 2\nu
\partial _{x_{3}}(\cdot )_{3}+\lambda \text{div}(\cdot )\right] _{\Omega } & 
0 & 0 \\ 
0 & 0 & 0 & 0%
\end{array}%
\right] 
\end{equation*}%
\begin{equation*}
+\left[ 
\begin{array}{cccc}
-\text{div}(\mathbf{U)(\cdot )} & 0 & 0 & 0 \\ 
0 & 0 & 0 & 0 \\ 
0 & 0 & -{{\mathring{A}}^{-1}\left\{ (\text{div}{{[U}_{1},U_{2}]{+\mathbf{%
U\cdot }\nabla )}\Delta }^{2}(\cdot)\right\} +}\mathbf{U}\mathbb{\cdot }%
\nabla (\cdot)+\Delta {{\mathring{A}}^{-1}\nabla }^{\ast }(\mathbb{\nabla
\cdot }(\mathbf{U}\mathbb{\cdot }\nabla (\cdot))\mathbb{)} & 0 \\ 
0 & 0 & 0 & 0%
\end{array}%
\right] 
\end{equation*}%
\begin{equation}
=L_{1}+L_{2}+B^{\ast }  \label{adj-AplusB}
\end{equation}%
Here, $\nabla ^{\ast }\in \mathcal{L}(L^{2}(\Omega ),[H^{1}(\Omega
)]^{^{\prime }})$ is the adjoint of the gradient operator $\nabla \in 
\mathcal{L}(H^{1}(\Omega ),L^{2}(\Omega ))$ and the domain of $(\mathcal{A%
}+B)^{\ast }|_{H_{N}^{\bot }}$ is given as 
\begin{equation*}
D((\mathcal{A}+B)^{\ast })\cap H_{N}^{\bot }=\{(p_{0},u_{0},w_{1},w_{2})\in
L^{2}(\mathcal{O})\times \mathbf{H}^{1}(\mathcal{O})\times H_{0}^{2}(\Omega
)\times L^{2}(\Omega )~:~\text{properties }\mathbf{(A^{\ast }.i)}\text{--}\mathbf{(A^{\ast }.vii)}~~\text{hold}%
\},
\end{equation*}%
where
\end{lemma}

\begin{enumerate}
\item $\mathbf{(A^{\ast }.i)}$ $\mathbf{U}\cdot \nabla p_{0}\in L^{2}(\mathcal{O})$

\item $\mathbf{(A^{\ast }.ii)}$ $\text{div}~\sigma (u_{0})+\nabla p_{0}\in \mathbf{L}^{2}(%
\mathcal{O})$ (So, $\left[ \sigma (u_{0})\mathbf{n}+p_{0}\mathbf{n}\right]
_{\partial \mathcal{O}}\in \mathbf{H}^{-\frac{1}{2}}(\partial \mathcal{O})$)

\item $ \mathbf{(A^{\ast }.iii)}$ $\Delta ^{2}w_{1}-\left[ 2\nu \partial
_{x_{3}}(u_{0})_{3}+\lambda \text{div}(u_{0})\right] _{\Omega
}-p_{0}|_{\Omega }\in L^{2}(\Omega )$

\item $\mathbf{(A^{\ast }.iv)}$ $\left( \sigma (u_{0})\mathbf{n}+p_{0}\mathbf{n}\right) \bot
~TH^{1/2}(\partial \mathcal{O})$. That is, 
\begin{equation*}
\left\langle \sigma (u_{0})\mathbf{n}+p_{0}\mathbf{n},\mathbf{\tau }%
\right\rangle _{\mathbf{H}^{-\frac{1}{2}}(\partial \mathcal{O})\times 
\mathbf{H}^{\frac{1}{2}}(\partial \mathcal{O})}=0\text{ \ in }\mathcal{D}%
^{\prime }(\mathcal{O})\text{\ for every }\mathbf{\tau }\in
TH^{1/2}(\partial \mathcal{O})
\end{equation*}

\item $\mathbf{(A^{\ast }.v)}$ The flow velocity component $u_{0}=\mathbf{f}_{0}+\widetilde{%
\mathbf{f}}_{0}$, where $\mathbf{f}_{0}\in \mathbf{V}_{0}$ and $\widetilde{%
\mathbf{f}}_{0}\in \mathbf{H}^{1}(\mathcal{O})$ satisfies
\begin{equation*}
\widetilde{\mathbf{f}}_{0}=%
\begin{cases}
0 & ~\text{ on }~S \\ 
w_{2}\mathbf{n} & ~\text{ on}~\Omega%
\end{cases}%
\end{equation*}%
\noindent (and so $\left. \mathbf{f}_{0}\right\vert _{\partial \mathcal{O}%
}\in TH^{1/2}(\partial \mathcal{O})$)

\item $\mathbf{(A^{\ast }.vi)}$ $[-w_{2}+\mathbf{U}\mathbb{\cdot }\nabla w_{1}+\Delta {{\mathring{A}}%
^{-1}\nabla }^{\ast }(\mathbb{\nabla \cdot }(\mathbf{U}\mathbb{\cdot }\nabla
w_{1})\mathbb{)}]\in H_{0}^{2}(\Omega ),$ \ (and so $w_{2}\in
H_{0}^{1}(\Omega )$)

\item $\mathbf{(A^{\ast }.vii)}$ $\int\limits_{\mathcal{O}}[\mathbf{U}\cdot \nabla p_{0}+$div$%
~(u_{0})]d\mathcal{O}$\newline
$+\int\limits_{\Omega }{{\mathring{A}}^{-1}}\left\{ (%
\text{div}{[U}_{1},U_{2}{]+\mathbf{U\cdot }\nabla )(\left[ p_{0}+2\nu
\partial _{x_{3}}(u_0 )_{3}+\lambda \text{div}(u_{0})\right] _{\Omega })}%
\right\} d\Omega $\newline
$-\int\limits_{\Omega}{{\mathring{A}}^{-1}\left\{ (\text{div}{{[U%
}_{1},U_{2}]{+\mathbf{U\cdot }\nabla )}\Delta }^{2}{w}_{1}\right\} }d\Omega $%
\newline
$+\int\limits_{\Omega}[\mathbf{U}\mathbb{\cdot }\nabla
w_{1}+\Delta {{\mathring{A}}^{-1}\nabla }^{\ast }(\mathbb{\nabla \cdot }(%
\mathbf{U}\mathbb{\cdot }\nabla w_{1})\mathbb{)]}d\Omega $\newline
$=0.$\newline
\end{enumerate}

\begin{proof}
Let $\varphi =\left[ p_{0},u_{0},w_{1},w_{2}\right] \in D(\mathcal{A}+B)\cap H_{N}^{\bot },$ $%
\widetilde{\varphi }=\left[ \widetilde{p}_{0},\widetilde{u}_{0},\widetilde{w}%
_{1},\widetilde{w}_{2}\right] \in D(\mathcal{A}+B)^{\ast }\cap H_{N}^{\bot }.
$ Then, we have%
\begin{equation*}
\left( \mathcal{A}\varphi ,\widetilde{\varphi }\right) _{\mathcal{H}}=-(%
\mathbf{U}\nabla p_{0},\widetilde{p}_{0})_{\mathcal{O}}-(\text{div}(u_{0}),%
\widetilde{p}_{0})_{\mathcal{O}}-(\nabla p_{0},\widetilde{u}_{0})_{\mathcal{O%
}}
\end{equation*}%
\begin{equation*}
+(\text{div}\sigma (u_{0}),\widetilde{u}_{0})_{\mathcal{O}}-\eta (u_{0},%
\widetilde{u}_{0})_{\mathcal{O}}-(\mathbf{U}\nabla u_{0},\widetilde{u}_{0})_{%
\mathcal{O}}
\end{equation*}%
\begin{equation*}
+(\Delta w_{2},\Delta \widetilde{w}_{1})_{\Omega }+(p_{0}|_{\Omega }-\left[
2\nu \partial _{x_{3}}(u_{0})_{3}+\lambda \text{div}(u_{0})\right] |_{\Omega
},\widetilde{w}_{2})_{\Omega }-(\Delta ^{2}w_{1},\widetilde{w}_{2})_{\Omega }
\end{equation*}%
\begin{equation*}
=(p_{0},\text{div}(\mathbf{U)}\widetilde{p}_{0})_{\mathcal{O}}+(p_{0},%
\mathbf{U}\nabla \widetilde{p}_{0})_{\mathcal{O}}-\left\langle u_{0}\cdot 
\mathbf{n,}\widetilde{p}_{0}\right\rangle _{\partial \mathcal{O}%
}+(u_{0},\nabla \widetilde{p}_{0})_{\mathcal{O}}
\end{equation*}%
\begin{equation*}
+(p_{0},\text{div}(\widetilde{u}_{0}))_{\mathcal{O}}-\left\langle p_{0}%
\mathbf{,}\widetilde{u}_{0}\cdot \mathbf{n}\right\rangle _{\partial \mathcal{%
O}}-(\sigma (u_{0}),\epsilon (\widetilde{u}_{0}))_{\mathcal{O}}
\end{equation*}%
\begin{equation*}
+\left\langle \sigma (u_{0})\cdot \mathbf{n},\widetilde{u}_{0}\right\rangle
_{\partial \mathcal{O}}-\eta (u_{0},\widetilde{u}_{0})_{\mathcal{O}}
\end{equation*}%
\begin{equation*}
+(u_{0},\text{div}(\mathbf{U)}\widetilde{u}_{0})_{\mathcal{O}}+(u_{0},%
\mathbf{U}\nabla \widetilde{u}_{0})_{\mathcal{O}}+(\Delta w_{2},\Delta 
\widetilde{w}_{1})_{\Omega }
\end{equation*}%
\begin{equation*}
-(\left[ 2\nu \partial _{x_{3}}(u_{0})_{3}+\lambda \text{div}(u_{0})\right]
|_{\Omega }-p_{0}|_{\Omega },\widetilde{w}_{2})_{\Omega }-(\Delta w_{1},\Delta \widetilde{w}%
_{2})_{\Omega }.
\end{equation*}%
Using the domain criterion $\textbf{(A.vi)},$ we then have from the above equality%
\begin{equation*}
\left( \mathcal{A}\varphi ,\widetilde{\varphi }\right) _{\mathcal{H}}=(p_{0},%
\text{div}(\mathbf{U)}\widetilde{p}_{0})_{\mathcal{O}}+(p_{0},\mathbf{U}%
\nabla \widetilde{p}_{0})_{\mathcal{O}}
\end{equation*}%
\begin{equation*}
-(w_{2}+\mathbf{U}\nabla w_{1},\widetilde{p}_{0})_{\Omega }+(u_{0},\nabla 
\widetilde{p}_{0})_{\mathcal{O}}+(p_{0},\text{div}(\widetilde{u}_{0}))_{%
\mathcal{O}}
\end{equation*}%
\begin{equation*}
-(\sigma (u_{0}),\epsilon (\widetilde{u}_{0}))_{\mathcal{O}}-\eta (u_{0},%
\widetilde{u}_{0})_{\mathcal{O}}+(u_{0},\text{div}(\mathbf{U)}\widetilde{u}%
_{0})_{\mathcal{O}}+(u_{0},\mathbf{U}\nabla \widetilde{u}_{0})_{\mathcal{O}}
\end{equation*}%
\begin{equation*}
+(w_{2},\Delta ^{2}\widetilde{w}_{1})_{\Omega }-(\Delta w_{1},\Delta 
\widetilde{w}_{2})_{\Omega }.
\end{equation*}%
Subsequently, integrating by parts in the third line of the last relation,
we get%
\begin{equation*}
\left( \mathcal{A}\varphi ,\widetilde{\varphi }\right) _{\mathcal{H}}=(p_{0},%
\text{div}(\mathbf{U)}\widetilde{p}_{0})_{\mathcal{O}}+(p_{0},\mathbf{U}%
\nabla \widetilde{p}_{0})_{\mathcal{O}}
\end{equation*}%
\begin{equation*}
-(w_{2}+\mathbf{U}\nabla w_{1},\widetilde{p}_{0})_{\Omega }+(u_{0},\nabla 
\widetilde{p}_{0})_{\mathcal{O}}+(p_{0},\text{div}(\widetilde{u}_{0}))_{%
\mathcal{O}}
\end{equation*}%
\begin{equation*}
+(u_{0},\text{div}\sigma (\widetilde{u}_{0}))_{\mathcal{O}}-\left\langle
u_{0},\sigma (\widetilde{u}_{0})\cdot \mathbf{n}\right\rangle _{\partial 
\mathcal{O}}-\eta (u_{0},\widetilde{u}_{0})_{\mathcal{O}}
\end{equation*}%
\begin{equation*}
+(u_{0},\text{div}(\mathbf{U)}\widetilde{u}_{0})_{\mathcal{O}}+(u_{0},%
\mathbf{U}\nabla \widetilde{u}_{0})_{\mathcal{O}}
\end{equation*}%
\begin{equation*}
+(w_{2},\Delta ^{2}\widetilde{w}_{1})_{\Omega }-(\Delta w_{1},\Delta 
\widetilde{w}_{2})_{\Omega }.
\end{equation*}%
Now, integrating by parts in the second line, and using again domain
criterion $\textbf{(A.vi)},$ we have%
\begin{equation*}
\left( \mathcal{A}\varphi ,\widetilde{\varphi }\right) _{\mathcal{H}}=(p_{0},%
\text{div}(\mathbf{U)}\widetilde{p}_{0})_{\mathcal{O}}+(p_{0},\mathbf{U}%
\nabla \widetilde{p}_{0})_{\mathcal{O}}
\end{equation*}%
\begin{equation*}
-(w_{2},\left[ \widetilde{p}_{0}+2\nu \partial _{x_{3}}(\widetilde{u}%
_{0})_{3}+\lambda \text{div}(\widetilde{u}_{0})\right] |_{\Omega })_{\Omega }
\end{equation*}%
\begin{equation*}
+(w_{1},(\text{div}[U_{1},U_{2}]+\mathbf{U}\nabla )\left[ \widetilde{p}%
_{0}+2\nu \partial _{x_{3}}(\widetilde{u}_{0})_{3}+\lambda \text{div}(%
\widetilde{u}_{0})\right] |_{\Omega })_{\Omega }
\end{equation*}%
\begin{equation*}
+(u_{0},\nabla \widetilde{p}_{0})_{\mathcal{O}}+(p_{0},\text{div}(\widetilde{%
u}_{0}))_{\mathcal{O}}+(u_{0},\text{div}\sigma (\widetilde{u}_{0}))_{%
\mathcal{O}}
\end{equation*}%
\begin{equation*}
-\eta (u_{0},\widetilde{u}_{0})_{\mathcal{O}}+(u_{0},\text{div}(\mathbf{U)}%
\widetilde{u}_{0})_{\mathcal{O}}+(u_{0},\mathbf{U}\nabla \widetilde{u}_{0})_{%
\mathcal{O}}
\end{equation*}%
\begin{equation}
+(w_{2},\Delta ^{2}\widetilde{w}_{1})_{\Omega }-(\Delta w_{1},\Delta 
\widetilde{w}_{2})_{\Omega }.  \label{adj-A}
\end{equation}%
Also we have%
\begin{equation}
\left( B\varphi ,\widetilde{\varphi }\right) _{\mathcal{H}}=-(\text{div}(%
\mathbf{U)}p_{0},\widetilde{p}_{0})_{\mathcal{O}}+(\Delta (\mathbf{U}\nabla
w_{1}),\Delta \widetilde{w}_{1})_{\Omega }.  \label{adj-0}
\end{equation}%
For the second term of the RHS of the above equality: for any $w_{1},%
\widetilde{w}_{1}\in H^{3}(\Omega )$%
\begin{equation*}
(\Delta (\mathbf{U}\nabla w_{1}),\Delta \widetilde{w}_{1})_{\Omega
}=\left\langle \frac{\partial }{\partial \nu }(\mathbf{U}\nabla
w_{1}),\Delta \widetilde{w}_{1}\right\rangle _{\partial \Omega}
\end{equation*}%
\begin{equation*}
-(\nabla (\mathbf{U}\nabla w_{1}),\nabla \Delta \widetilde{w}_{1})_{\Omega }
\end{equation*}%
\begin{equation*}
=\left\langle (\mathbf{U\cdot \nu )}\Delta w_{1},\Delta \widetilde{w}%
_{1}\right\rangle _{\partial \Omega}-(\nabla (\mathbf{U}\nabla
w_{1}),\nabla \Delta \widetilde{w}_{1})_{\Omega }
\end{equation*}%
where we have used the fact that $w_{1}=\frac{\partial w_{1}}{\partial \nu }%
=0$ and this yields%
\begin{equation*}
\frac{\partial }{\partial \nu }(\mathbf{U}\nabla w_{1})=(\mathbf{U\cdot \nu )%
}\frac{\partial ^{2}w_{1}}{\partial \nu }=(\mathbf{U\cdot \nu )(}\Delta
w_{1}|_{\partial \Omega}\mathbf{).}
\end{equation*}%
Then%
\begin{equation*}
(\Delta (\mathbf{U}\nabla w_{1}),\Delta \widetilde{w}_{1})_{\Omega
}=\left\langle \Delta w_{1},\frac{\partial }{\partial \nu }(\mathbf{U}\nabla 
\widetilde{w}_{1})\right\rangle _{\partial \Omega}-(\nabla (%
\mathbf{U}\nabla w_{1}),\nabla \Delta \widetilde{w}_{1})_{\Omega }
\end{equation*}%
\begin{equation*}
=(\Delta w_{1},\Delta (\mathbf{U}\nabla \widetilde{w}_{1}))_{\Omega
}+(\nabla \Delta w_{1},\nabla (\mathbf{U}\nabla \widetilde{w}_{1}))_{\Omega
}-(\nabla (\mathbf{U}\nabla w_{1}),\nabla \Delta \widetilde{w}_{1})_{\Omega }
\end{equation*}%
\begin{equation}
=(\Delta w_{1},\Delta (\mathbf{U}\nabla \widetilde{w}_{1}))_{\Omega
}+(\Delta w_{1},\nabla ^{\ast }[\nabla (\mathbf{U}\nabla \widetilde{w}%
_{1})])_{\Omega }-(\nabla (\mathbf{U}\nabla w_{1}),\nabla \Delta \widetilde{w%
}_{1})_{\Omega }  \label{adj-1}
\end{equation}%
where $\nabla ^{\ast }\in \mathcal{L}(L^{2}(\Omega ),[H^{1}(\Omega
)]^{^{\prime }})$ is the adjoint of the gradient operator $\nabla \in 
\mathcal{L}(H^{1}(\Omega ),[L^{2}(\Omega )]).$ To continue with the third
term on RHS of (\ref{adj-1}):%
\begin{equation*}
-(\nabla (\mathbf{U}\nabla w_{1}),\nabla \Delta \widetilde{w}_{1})_{\Omega
}=(\mathbf{U}\nabla w_{1},\Delta ^{2}\widetilde{w}_{1})_{\Omega }
\end{equation*}%
\begin{equation*}
=-(w_{1},\left\{ \text{div}[U_{1},U_{2}]+\mathbf{U}\nabla \right\} \Delta
^{2}\widetilde{w}_{1})_{\Omega }
\end{equation*}%
\begin{equation}
=-(\Delta w_{1},\Delta {{\mathring{A}}^{-1}}\left\{ \text{div}[U_{1},U_{2}]+%
\mathbf{U}\nabla \right\} \Delta ^{2}\widetilde{w}_{1})_{\Omega }
\label{adj-2}
\end{equation}%
If we take into account (\ref{adj-2}) in (\ref{adj-1}) and invoke the
biharmonic operator with clamped homogeneous boundary conditions we take%
\begin{equation*}
(\Delta (\mathbf{U}\nabla w_{1}),\Delta \widetilde{w}_{1})_{\Omega
}=-(\Delta w_{1},\Delta {{\mathring{A}}^{-1}}\left\{ \text{div}[U_{1},U_{2}]+%
\mathbf{U}\nabla \right\} \Delta ^{2}\widetilde{w}_{1})_{\Omega }
\end{equation*}%
\begin{equation}
+(\Delta w_{1},\Delta (\mathbf{U}\nabla \widetilde{w}_{1}))_{\Omega }+(\Delta
w_{1},\Delta \lbrack \Delta {{\mathring{A}}^{-1}}\nabla ^{\ast }[\nabla (%
\mathbf{U}\nabla \widetilde{w}_{1})]])_{\Omega }.  \label{adj-3}
\end{equation}%
Now, considering (\ref{adj-3}) in (\ref{adj-0}) and combining the result
with (\ref{adj-A}) gives the adjoint operator given in (\ref{adj-AplusB}) and completes the proof of Lemma \ref{adj}.
\end{proof}

In order to establish the wellposedness result, one of the key tools that we use in our proof is the invertibility criterion of a linear, closed operator which we recall in the following lemma \cite[pg.102, Lemma 3.8.18]{pazy}:  
\begin{lemma}
\label{pazy}  Let $L$ be a linear and closed operator from the Hilbert
space $H$ into $H$. Then $L^{-1}\in \mathcal{L}(H)$ if and only if $R(L)$ is
dense in $H$ and there is an $m>0$ such that%
\begin{equation*}
\left\Vert Lf\right\Vert \geq m\left\Vert f\right\Vert \text{ \ for all \ }%
f\in D(L).
\end{equation*}
\end{lemma}

\section{Acknowledgement}

The author would like to thank the National Science Foundation, and
acknowledge her partial funding from NSF Grant DMS-1616425 and NSF Grant
DMS-1907823.

\end{document}